\documentclass[12pt]{amsart}
\usepackage[english]{babel}
\usepackage{amsfonts,latexsym}
\usepackage{graphicx}
\usepackage{amsmath}
\usepackage{amsthm}
\usepackage{dsfont}
\usepackage{amssymb}
\usepackage{mathrsfs}
\usepackage{tikz}
\usetikzlibrary{calc,patterns}
\usetikzlibrary{arrows}
\usepackage{amsmath,amsthm,amssymb,mathtools,bm}
\usepackage{microtype}
\usepackage{subfig}
\usepackage{enumitem}
\usetikzlibrary{arrows.meta}
\usetikzlibrary{decorations.markings}
\usepackage{float}
\usepackage{hyperref}
\usetikzlibrary{calc}
\AtBeginDocument{ 
   \def\MR#1{}
}
  \newcounter{constant} 
  \newcommand{\newconstant}[1]{\refstepcounter{constant}\label{#1}} 
  \newcommand{\useconstant}[1]{c_{\ref{#1}}}

\usepackage[framemethod=TikZ]{mdframed}
\mdfdefinestyle{MyFrame1}{  linecolor=gray!40!white,   backgroundcolor=gray!9!white}
     \mdfdefinestyle{MyFrame2}{%
    linecolor=gray!40!white,
    backgroundcolor=blue!4!white}

\newcommand{\eqa}{\begin{eqnarray}}
\newcommand{\ena}{\end{eqnarray}}
\newcommand{\eq}{\begin{equation}}
\newcommand{\en}{\end{equation}}
\newcommand{\eqs}{\begin{eqnarray*}}
\newcommand{\ens}{\end{eqnarray*}}
\def\ti{\to\infty}

\def\X{\mathbf{X}}

\newcommand{\R}     {\mathbb{R}}
\newcommand{\Z}     {\mathbb{Z}}
\newcommand{\N}     {\mathbb{N}}
\renewcommand{\P}   {\mathbb{P}}
\newcommand{\E}     {\mathbb{E}}
 \newcommand{\floor}[1]{\left\lfloor #1 \right\rfloor}
\newcommand{\ceil}[1]{\left\lceil #1 \right\rceil}

\newcommand{\bds}{\begin{displaystyle}}
\newcommand{\eds}{\end{displaystyle}}

\newtheorem{theorem}{Theorem}[section]

\newtheorem{lemma}[theorem]{Lemma}
\newtheorem{proposition}[theorem]{Proposition}
\newtheorem{corollary}[theorem]{Corollary}

\newtheorem{definition}[theorem]{Definition}

\newcommand{\eps}{\varepsilon}

\def\1{{\mathchoice {1\mskip-4mu\mathrm l}      
{1\mskip-4mu\mathrm l}
{1\mskip-4.5mu\mathrm l} {1\mskip-5mu\mathrm l}}}
\newcommand{\ssup}[1] {{{\scriptscriptstyle{({#1}})}}}
\def\comment#1{}

\begin{document}
\title[ORRW and Self-Interacting RWs  beyond exchangeability I]{Once-Reinforced  and Self-Interacting Random Walks beyond exchangeability.
}
\author[A. Collevecchio and P. Tarr\`es]{Andrea Collevecchio and Pierre Tarr\`es\\[1ex] 
\tiny \it  Dedicated to Vladas Sidoravicius}

\date{}

\begin{abstract}
We present the first rigorous quantitative analysis of once-reinforced random walks (ORRW) on general graphs, based on a novel change of measure formula.~This enables us to prove large deviations estimates for the range of the walk to have cardinality of the order $N^{d/(d+2)}$ in dimension larger than or equal than two. We also prove that ORRW is transient on all non-amenable graphs for small reinforcement.~Moreover, we study the shape of oriented ORRW on euclidean lattices. 

We also provide a new approach to the study of general self-interacting random walk, which we apply to random walk in random environment, reinforced processes on oriented graphs, including the directed ORRW. 
\end{abstract}
\maketitle


\section{Introduction}\label{intro}
 In 1990, Burgess Davis \cite{Dav90} introduced a simple reinforcement model, called Once-Reinforced
Random Walk (ORRW), which can be roughly described as follows. To each edge of a locally finite graph
assign a positive weight. The process jumps to nearest neighbors. An edge is traversed with a
probability proportional to its weight. The very first time an edge is traversed its weight changes, and then
remains unaltered.
Despite its deceptively simple definition, ORRW turns out to be difficult to analyze, and many conjectures have
been made on its behaviour on the multidimensional lattice $\Z^d$.
\\
The ORRW is part of the larger class of Reinforced random walks introduced by Coppersmith and Diaconis \cite{CD87}, and the first model
to be considered was the so-called Linearly Edge-Reinforced Random Walk (LERRW), where each time an edge is
traversed by the process, its weight is increased by a constant. Several results are available for LERRW (see \cite{ST15}  \cite{DST15} \cite{ACK14} \cite{RM.06}) where the authors use partial exchangeability. The latter property is
not satisfied by ORRW, and this explains why this process seems harder to study when
compared to LERRW.

 The ORRW has been extensively studied on $\Z$ and other trees, where ruin probabilities can be computed.~ Durrett, Kesten and  Limic proved transience of ORRW on regular trees \cite{DKL02}.
 Collevecchio  \cite{Coll06} proved transience of ORRW on supercritical Galton-Watson trees, while Pfaﬀelhuber and  Stiefel  \cite{PS21} studied the range of ORRW on $\Z$. Kious and Sidoravicius \cite{KS18} showed that ORRW exhibits a phase transition between recurrence/transience on certain polynomial trees.  A complete analysis of ORRW on trees that grow polynomially was provided in \cite{CKS20},  where the authors identify the critical value of the reinforcement for recurrence/transience.

Very few results are available when it is defined on graphs with cycles.
Thomas Sellke proved that the process is recurrent on  ladders $\Z \times \{1, 2, \ldots, d\}$ for small reinforcement and  Vervoort  \cite{Ve00} and  Kious, Schapira and Singh \cite{KSS18} proved recurrence for large $\delta$. Notice that there is no complete picture for recurrence/transience on the ladder.

There are many long-standing conjectures regarding the behaviour  of ORRW on $\Z^d$ and non-amenable graphs. 

A model that mimics features of ORRW was introduced and studied in \cite{DGHS21} where the authors
write:\\
“It is conjectured that for both ORRW and OERW\footnote{Stands for origin-excited random walk which we do not cover in this paper.} the evolution leads to the formation of
an asymptotic shape as time goes to infinity, but there is no clear vision on how to attack the
problem.“\\
Vincent Beffara, in his Habilitation  \cite{Beffara} (page 54,  conjecture 7), conjectured the following for ORRW on $\Z^2$. \\
''There exists a positive constant $b_{0} > 0$  (such that, whenever $b > b_{0}$, the process $(X_t)$ is recurrent\footnote{the larger is $b$ the larger is the reinforcement in Beffara notation. We use a different parametrization in our paper.}. Moreover, if $K_t$ is the set of points visited by time $t$, then almost surely
\[
|K_t| \approx t^{2/3}, \qquad \operatorname{diam} K_t \approx t^{1/3},
\]
and the rescaled set $K_t / \operatorname{diam} K_t$ converges in probability, in the Hausdorff topology, to a deterministic asymptotic shape $K_\infty$.
If the conjecture holds, it is very likely that in fact $b_{0} = 0$, i.e.\ that arbitrarily small reinforcement is sufficient to place the process in this sub-diffusive regime."\\

In Section \ref{su:orrw}, we provide a large deviations type  bound for the range of ORRW on $\Z^d$, with $d\ge2$, to have
cardinality of the order $t^{d/(d+2)}$.  Moreover, we prove that ORRW on non-amenable graphs  is  transient for all  small enough reinforcements. 

Along the way, we prove results that have some independent interest, even for simple random
walks (see Corollary 1.7 and Section 2), and negatively reinforced random walks, which are weakly
self-avoiding walks. 
Based on the technology  introduced in this paper, we believe we can show that the shape theorem  does not hold for  ORRW on $\Z^2$ for all small positive reinforcement: this part of the project is at a very advanced stage but still unfinished.

Our second set  of  results in Section \ref{su:gltt} concerns local time theorems for a large class of self-interacting processes, which includes random walks in random environment, general directed reinforced processes, which we apply in particular to the directed ORRW.

\subsection{Once-Reinforced Random Walks on General Graphs}
\label{su:orrw}
 We consider processes that take values  on  the vertices of  a locally bounded rooted graph $G= (V, E)$.  They start from the root, which we denote by $0$. In the case of $\Z^d$ we assume that the origin is the root. Initially to
each edge is assigned weight one.~When the process traverses an edge for the first time,   its weight is
upgraded to $1/a$, and then never changes again.
This walk can be embedded in a continuous time process $\X = (X_t)_{t\ge 0}$ which we define rigorously as
follows. Set $X_0 =  0$, i.e. the root. The process $\X$ jumps to nearest neighbors vertices. Suppose
we defined $(X_u)_{u\le t}$, and let
$$
\mathfrak{C}_t := \{e \in E\colon \mbox{ edge $e$ is traversed by $\X$ by time $t$}\},
$$
i.e. the edge range of the process by time $t$.~Define $\mathcal{F}_t = \sigma(X_u\colon u \le t).$  We have that on $\{X_t = x\}$,
$$ \P(X_{t + h} = j\;|\; X_t =i, \mathcal{F}_t) =  \1_{\{\{i,j\} \in \mathfrak{C}_t\}} h    + a \1_{\{\{i,j\} \notin \mathfrak{C}_t\}} h   + o(h),
$$
where $y \sim x$. The probability to have more than one jump in the interval $(t, t + h)$ is $o(h)$. Denote by $\mathfrak{R}_t$ the
 vertex range of the process $\X$ by time $t$, i.e.
$$
\mathfrak{R}_t :=  \{x \in V \colon \exists s \in [0,t] \mbox{ such that } X_s = x\}. 
$$
 From now on, we use  $\P^{\ssup  a}$ to denote the probability measure on this space relative to ORRW($a$) and $\P$ is used for  the probability measure
relative to continuous time simple symmetric random walk. 
The process $\X$ is right-continuous. 
Let $(\tau_i)_i$ be the times when the process $\X$ jumps, i.e. $\tau_0 :=0$ and $\tau_n := \inf\{t > \tau_{n-1}\colon X_t \neq X_{\tau_{n-1}}\}$.
For $x \in \R^d$, let $\|x\| = \|x\|_2$ be its euclidean distance from the origin, and ${\rm Ball}(0, r) := \{x \in \R\colon \|x\| \le r\}$. Let  $\lambda_d$ the principal eigenvalue of the operator $-\Delta/2$ on ${\rm Ball}(0, 1)$ with Dirichlet boundary conditions. Set 
 $\omega_d$ to be the volume of ${\rm Ball}(0, 1)$, and
$$
\psi_d :=\frac{d + 2}2
\left(\frac{
2\lambda_d}d \right)
\omega^{2/(2+d)}_d.
$$
Our first result is a large deviations-type of  bound. For any finite set $A$, we use either $|A|$ or  ${\rm card}(A)$ to
denote its cardinality.
\begin{mdframed}
[style=MyFrame2]
\begin{theorem}\label{th:main}
Consider ORRW with parameter $a > 0$.  Choose $p > 1$ such that $(1-a)p < 1$.
Set
\begin{equation}\label{eq:defnu}
\nu(d, a, u, p) :=\begin{cases} -u\left(\frac{p-1}p -
\frac{2d}p
\log(1 - p(1 - a)) \right)+
\frac{p - 1}p
\psi_d, &\qquad \mbox{if $a \in (0,1)$}\\
 -u\left(d (\log a)+1\right) + \psi_d,  &\qquad \mbox{if $a \ge 1$}.
\end{cases}
\end{equation}
We have, for $u >0$, 
\begin{equation}\label{eq:largedeviR}
\P^{\ssup a}\left(|\mathfrak{R}_{\tau_N}| \le uN^{d/(d+2)}\right) \le \exp\{-(\nu(d, a, u, p) + o(1))N^{d/(d+2)}\},
\end{equation}
where $o(1)$  stands for a quantity which depends on $a, p,d$ and $u$ and approaches zero as $N \ti$, uniformly in $a$. 
It is worth noting that for any $a \in (0,1)$ and $p>1$ satisfying $(1-a)p < 1$, one has  $\sup_{u >0} \nu(d, a, u, p)>0$.
\end{theorem}
\end{mdframed}
\vspace{0.4cm}
We also analyse the behaviour of ORRW on non-amenable graphs, i.e.~graphs $G= (V, E)$ such that  there exists a constant $\gamma_{G}>0$, called the Cheeger constant, which satisfies the following. For any finite subgraph $G_0 \subset G$, one has 
$$
{\rm Card}(\partial^{out} G_0) \ge \gamma_{G} \cdot {\rm Card}(G_0),
$$
where $\partial^{out} G_0$ is the set of vertices in $G^c_0$ that have distance one from $G_0$.
\vspace{0.5cm}

\begin{mdframed}
[style=MyFrame2]
\begin{theorem}\label{th:maintamen}
For any non-amenable graph $G$, the following holds. \\
\noindent a) There exists $a_{G} \in (0,1)$ such that    for all $a  \in (a_{G}, 1)$ there exists $C>0$ and $\beta>0$ depending on $a$, satisfying 
$$
\P^{\ssup a}(X_{\tau_n} = 0) \le C {\rm e}^{- \beta n}.
$$
The latter implies that the process is transient for those choices of $a$.\\
\noindent b)  For all $a \in (0,1)$,
$$
\P^{\ssup a}\left(\{\X \mbox{ is recurrent}\} \cap \Big\{\lim_{t \ti} \frac{|\mathfrak{C}_t|}t  =0  \Big\}\right) =0.
$$
 \end{theorem}
\end{mdframed}
Theorem~\ref{th:maintamen} b) states that  if ORRW($a$), for any choice of $a\in(0,1)$ was to be recurrent, then the range would not grow sub-linearly, i.e. there would be a sequence of times where the range has cardinality comparable to the  time. \\
The following result is valid for any transient graph, i.e. graphs where simple random walk is transient. 
\vspace{0.2cm}
\begin{mdframed}
[style=MyFrame2]
\begin{theorem}\label{th:main8}
Let $G$ be a transient graph. Let $S$ be the first return time to the starting point.  One has 
$$
\limsup_{a \uparrow 1}\, \E^{\ssup a}[S] = \infty.
$$
 \end{theorem}
\end{mdframed}
Advanced work in progress of the authors using the same set of  techniques seems to achieve a non-shape theorem for $\mathbb{Z}^2$ for all $a$ close enough to one. 

\noindent Our proofs of Theorems~\ref{th:main} and \ref{th:main4} use   a novel "polymer" representation of ORRW, i.e. a change of measure with respect to the simple random walk measure  which we believe is of independent interest. 
For each vertex $x$ and edge $e$, we use $x \sim e$ to denote that an edge $e$ is incident to $x$.  For an edge $e$ we denote by $e^+$ and $e^-$ its endpoints, using an arbitrary choice. Let $\mathbb{L}_e$ be the first time edge $e$ is traversed by the process, i.e.
 $$
 \mathbb{L}_e := \inf\{t \colon \mbox{either  $X_{t_-} = e^+$ and $X_{t} = e^-$ or $X_{t_-} = e^-$ and $X_{t} = e^+$}\}.
 $$
\begin{mdframed}
[style=MyFrame2]
\begin{proposition}\label{pr:changem}
Consider ORRW($a$) on a graph $G= (V, E)$. Fix $n \in \N := \{0, 1, 2, \ldots\}$. For any event $A \in \mathcal{F}_{\tau_n}$, one has
$$
\P^{\ssup a}(A)= \E\left[\exp\left\{(1-a)\sum_{e \in E} T_e(\tau_n) + (\log a) |\mathfrak{C}_{\tau_n}| \right\}\1_A\right]
,$$
where $T_e(\tau_n)$ is the time the process spent adjacent to edge $e$ before time   $\mathbb{L}_e\wedge \tau_n$, i.e.
\begin{equation}\label{eq:dfT}
T_e(\tau_n) := \int_{0}^{\mathbb{L}_e\wedge \tau_n} \1_{X_u \sim e} {\rm d} u.
\end{equation}
\end{proposition}
\end{mdframed}
\subsection{Local time theorems for general self-interacting random walks}
\label{su:gltt}
Next we focus on studying the distribution of local time profile for a general class of processes.  Our work generalises the work of Kious, Huang, Sidoravicius and Tarr\`es,  \cite{Pierre1} and we adopt the same notation used in that paper.
Let $G = (V, E)$ be a  connected  graph which is either finite or infinite countable.  We  assume that $G$ is locally finite. It has no multiple edges and  it is rooted. We denote by $0$ its root.  Obtain $\vec{E}$  by replacing each edge in $E$ with two oriented edges, joining the same pair of neighbors in both directions. 
Fix a designated vertex $i_1$, which is the end point of a path. For $G' \subset G$, which contains $i_1$, denote by $\vec{\mathcal{T}}_{i_1}(G')$ the collection of oriented spanning trees $\vec{T}$ of $G'$ rooted at $i_1$. These spanning trees are oriented, and the orientation points from the leaves towards $i_1$.  In this scenario, the root $i_1$ is  the unique vertex from which no edge emanates in the spanning tree. We denote by $\delta_i(j) = \mathbf{1}_{\{i = j\}}$ the Kronecker delta.
Let $\mathcal{I}$ be the set of \emph{currents} on the graph, i.e.,
\[
\mathcal{I} := \left\{ b \in \mathbb{Z}^{\vec{E}} : b_{j, i} = -b_{i, j}, \quad i, j \in V \mbox{ with } i \sim j \right\}.
\]
For any $b \in \mathbb{Z}^{\vec{E}}$ and $i \in V$, let $b_i := \sum_{j \sim i} b_{i, j}$. If $b \in \mathcal{I}$, then $b_i$ can be interpreted as the divergence of $b$ at site $i$.
For any $k \in \mathbb{N}^{\vec{E}}$, let $b(k) \in \mathcal{I}$ be defined by $b(k)_{i,j} = k_{i, j} - k_{j,i}$. For any $b \in \mathcal{I}$ and any oriented spanning tree $\vec{T}$ of a connected subset $G'$ of $G$, define the adjusted current $\tilde{b}$ by:
\[
\tilde{b}_{i,j}(\vec{T}) = b_{i, j} - \mathbf{1}_{\{(i, j) \in \vec{T}\}} + \mathbf{1}_{\{(j, i) \in \vec{T}\}}, \quad \mbox{ for } \quad (i, j) \in \vec{E}.
\]
For any $\sigma > 0$, possibly a stopping time,  and any right-continuous path $x = (x(t))_{t \geq 0}$, define $\ell(x, \sigma) \in(0, \infty)^V$ as the vector of local times at time $\sigma$, that is:
\[
\ell(x, \sigma)_i = \int_0^\sigma \mathbf{1}_{\{x(s) = i\}} \, ds, \quad i \in V.
\]
Define $k(x, \sigma) = (k_{i, j}(x, \sigma))_{(i,j) \in \vec{E}}$ as the vector of oriented edge crossings up to time $\sigma$, i.e. 
\begin{equation}\label{eq:K}
k_{i, j}(x, \sigma) = \left| \{ t \leq \sigma : x_{t^-} = i, x_t = j \} \right|.
\end{equation}
Let $\vec{T}(x, \sigma)$ be the last-exit tree of the path $x$ on the interval $[0, \sigma]$, defined as the set of directed edges corresponding to the last departures from each visited vertex, except the terminal vertex $x(\sigma)$. That is, $(i, j) \in \vec{T}(x, \sigma)$ if there exists $t \in (0, \sigma]$ such that $(x_{t^-}, x_t) = (i, j)$ and $x(s) \ne i$ for all $s \geq t$. Let 
\begin{equation}\label{eq:defH}
h_{i,j}(\vec{T}) := \begin{cases} 1 \qquad \mbox{if $(i,j) \in \vec{T}$}\\
 0 \qquad \mbox{otherwise}.
 \end{cases}
 \end{equation}
 \begin{mdframed}[style=MyFrame1]
 \begin{definition}
 Fix a function $f \colon (0, \infty)^2 \rightarrow (0, \infty)$.
Consider a point process $(Y^{\ssup f}_t)_t$ which takes values on $\N$, is right-continuous, and is parametrised by $f$.  It satisfies $Y^{\ssup f}_{t+h} - Y^{\ssup f}_t \in \N$ and
$$
\P(Y^{\ssup f}_{t+h} - Y^{\ssup f}_t =1\;|\; Y^{\ssup f}_t) = f(Y^{\ssup f}_t, t) h + o(h), \qquad \P(Y^{\ssup f}_{t+h} - Y^{\ssup f}_t \ge 2\;|\; Y_t) =  o(h).
$$
Define 
\begin{equation}
\mathcal{P}(f, n , t) := \P(Y^{\ssup f}_t = n), \qquad \mbox{and} \qquad \mathcal{P}^*(f, n , t) {\rm d} t := \P(Y^{\ssup f}_{t^-} = n-1, Y^{\ssup f}_{t+dt} - Y^{\ssup f}_t =1).
\end{equation}
Set 
$$
\hat{\mathcal{P}}(f,  t,  z) := \sum_{n = 0}^\infty \mathcal{P}(f, t, n) {\rm e}^{2 \pi i n z}, \qquad \mbox{and} \qquad \hat{\mathcal{P}}^*(f,  t,  z) := \sum_{n = 1}^\infty \mathcal{P}^*(f, t, n) {\rm e}^{2 \pi i n z}.
$$
\end{definition}
\end{mdframed}
\begin{mdframed}[style=MyFrame2]
\begin{theorem}\label{th:supermainth} Let $\X$ be a nearest neighbor process  on a locally finite graph $G$ which satisfies the following property.    For any pair of neighbors $(i,j)$, there exists a function $f_{i,j}  \colon (0, \infty)^2 \rightarrow (0, \infty)$, such that on $X_t =i$ one has 
\begin{equation}
\P(X_{t+h} = j \;|\; \mathcal{F}_t) = f_{i,j}(k_{i,j}(\X, t), \ell_i(\X, t)) \delta_{i\sim j} h  +o(h),
\end{equation}
where  $\mathcal{F}_t = \sigma(X_u\colon u \in [0,t])$and $\delta_{i\sim j}$ is one if $i\sim j$ and is otherwise  zero. Fix  a connected subgraph $G' \subset G$, where $G' = (V', E')$ and $0 \in V'$. Fix $i_1 \in V'$ and a local time profile $\ell \in (0,\infty)^{V'}$ with $\sum_i \ell_i = t$ and $k \in \N^{E'}$ such that  $b(k)_i = \delta_{0}(i) - \delta_{i_1}(i)$ for all $i \in V'$.  Finally let $\vec{T}$ be an oriented  spanning tree of $G'$ rooted at $i_1$. One has
\begin{equation}\label{eq:loctim-1}
\begin{aligned}
\mathbb{P} \big(&k(\X, t) = k,\; \ell(\X, t) \in (\ell, \ell+ d\ell),\; \vec{T}(\X, t) = \vec{T} \big)\\
&= \prod_{(i,j) \in \vec{E'}\setminus{\vec T}} \mathcal{P}(f_{i,j}, k_{i,j}, \ell_i)\prod_{(i,j) \in {\vec T}} \mathcal{P}^*(f_{i,j}, k_{i,j}, \ell_i) m_t({\rm d}\ell),
\end{aligned}
\end{equation}
where $m_t$ is the lebesgue measure on the simplex.
Moreover 
\begin{equation}\label{eq:loctim-1.01}
\begin{aligned}
\mathbb{P} \big(&\ell(\X, t) \in (\ell, \ell+ d\ell),\; \vec{T}(\X, t) = \vec{T} \big)\\
&= \frac{1}{(2\pi)^{|V'|-1}}\int_{[0, 2\pi]^{V'\setminus\{i_1\} }} \prod_{(i,j) \in \vec{E}'\setminus{\vec T}} \hat{\mathcal{P}}(f_{i,j}, \ell_i, x_i - x_j)\prod_{(i,j) \in {\vec T}} \hat{\mathcal{P}}^*(f_{i,j}, \ell_i, x_i - x_j) {\rm d} x,
\end{aligned}
\end{equation}
where we set $x_{i_1} =0$.
\end{theorem}
\end{mdframed}
\vspace{0.3cm}
Below, we discuss few examples where Theorem~\ref{th:supermainth} can be applied. We focus on random walks in a random environment, general oriented reinforced processes, with an emphasis on oriented-once reinforced random walks. 

\subsubsection{Random walks in a random environment}

For each fixed environment $\omega$ and fixed vertex $v_0\in V$, the random walk in environment $\omega$ starting from $0$ is the nearest-neighbour Markov chain $\mathbf X= (X_n)_{n\ge0}$ taking values on $V$ with transition law $P_{\omega,v}$ given by $P_{\omega}(X_0=0)=1$ and
$$P_{\omega}(X_{t + h}=v|X_t=u)= \omega_{u,v} h + o(h).$$
We call $\P_{\omega}$ the \textit{quenched law} of $\mathbf X$. Let
$$\mathsf{P}(\cdot)=\int_{\Omega}P_{\omega}(\cdot) \text{d}\P(\omega)$$
which defines a probability measure on the space of nearest neighbour trajectories on $G$. We call $\mathsf P$ the \textit{annealed law} of $\mathbf X$. We denote by $\E$, $E_{\omega}$
 and $\mathsf E $ the expectations corresponding to the probability measures $\P$, $P_{\omega}$ and $\mathsf P$ respectively. 
\vspace{.6cm}
\begin{mdframed}[style=MyFrame2]
\begin{theorem}\label{th:RWREdirec}  Suppose that $G$ is a  regular graph. Define $\Psi \colon \N^\Delta \rightarrow (0, \infty)$, as 
$$ \Psi(n_1, n_2, \ldots, n_\Delta) := \E\left[\prod_{i\sim 0} \omega(0, i)^{n_i}\right].$$
Fix a connected subgraph $G' \subset G$, where $G' = (V', E')$, with $0 \in V'$ and a designated vertex $i_1 \in V'$. Fix a local time profile $\ell \in (0,\infty)^{V'}$ with $\sum_i \ell_i = t$ and $k \in \N^{E'}$ such that  $b(k)_i = \delta_{0}(i) - \delta_{i_1}(i)$ for all $i \in V'$.  Finally let $\vec{T}$ be an oriented  spanning tree of $G'$ rooted at $i_1$. One has
\begin{equation}\label{eq:loctim}
\begin{aligned}
\mathsf{P} \big(&k(\X, t) = k,\; \ell(\X, t) \in (\ell, \ell+ d\ell),\; \vec{T}(\X, t) = \vec{T} \big)  \\
&=  {\rm e}^{-t} \ell_{i_1}
 \left(\prod_{i\in V'} \ell_{i}^{k_{i} - 1} \Psi((k_{i,j}))_{j \colon j \sim i}) \prod_{j \colon j \sim i} \frac{1}{(k_{i,j} - h_{i,j}(\vec{T}))!}\right) m_t({\rm d} \ell).
\end{aligned}
\end{equation}
\end{theorem}
\end{mdframed}
Random walks in i.i.d. Dirichlet random environments form a particularly tractable and rich subclass of RWRE models, in which the transition probabilities at each site are independently drawn from a Dirichlet distribution. This setting offers a natural Bayesian framework for modeling random transitions on graphs, and it arises in various contexts, including statistical mechanics. For an accessible introduction, see \cite{enriquez2009dirichlet} and the survey by \cite{sabot2015survey}. In our case the Dirichlet structure allows for explicit calculations of the distribution of local time profile. 
\begin{mdframed}[style=MyFrame2] 
\begin{theorem}\label{th:DRW}Consider a  RWRE  $\X$ defined on a regular graph (possibly infinite) with degree $\Delta$ and  an i.i.d. Dirichlet random environment. More precisely $(\omega(0, j)_{j \colon j \sim 0}$ is a Dirichlet distribution whose density equals $\Gamma(\Delta)$ on the $(\Delta -1)$-simplex. 
Fix a connected subgraph $G' \subset G$, where $G' = (V', E')$ and $0 \in V'$. Fix  $i_1\in V'$ and a local time profile $\ell \in (0,\infty)^{V'}$ with $\sum_i \ell_i = t$ and $k \in \N^{E'}$ such that  $b(k)_i = \delta_{0}(i) - \delta_{i_1}(i)$ for all $i \in V$.  Let $\vec{T}$ be an oriented  spanning tree of $G'$ rooted at $i_1$. Denote by $\vec{\mathcal T_{i_1}}(G')$ the collection of oriented spanning trees of the graph $G'$ rooted at $i_1$.
 One has
\begin{equation}\label{eq:loctim1011}
\begin{aligned}
\mathsf{P} \big(&k(\X, t) = k,\; \ell(\X, t) \in (\ell, \ell+ d\ell), X_t =i_1\big) \\
&= {\rm e}^{-t} \Gamma(\Delta)^{|V'|}  \left(\prod_{i \in V'} \frac{\ell_i^{k_i -1}}{\Gamma(k_{i} +\Delta)}\right) \ell_{i_1} \left(\sum_{\vec{T} \in \mathcal{T}_{i_1}(G')} \prod_{(i,j) \in \vec{T}} k_{i,j}\right) m_t({\rm d}\ell).
\end{aligned}
\end{equation}
\end{theorem}
\end{mdframed}
Note that by the Matrix Tree Theorem the sum in the bracket does not depend on $i_1$, since it can be written as a principal minor of the matrix with off-diagonal coefficient $- k_{i, j}$ and $ \sum_{r} k_{i, r}$ on the diagonal.

\subsubsection{General reinforced random walks on oriented graphs}
We define a continuous time, right-continuous process, called  \lq directed\rq\  edge reinforced random walk (dERRW)  $\X = (X_t)_{t \geq 0}$ started at $0 \in V$, as follows. It takes values in $V$ and jumps to nearest neighbors.   Fix a reinforcement function $f \colon \N \rightarrow (0, \infty)$ and denote by $\vec{\mathbb{P}}^{\ssup f}$ the measure associated to the process.  Recall that  \( k_{i,j}(\X, t) \) denotes the number of times the directed edge \( (i,j) \in \vec{E} \) has been traversed up to time \( t \).
On the event $\{X_t =i\}$, one has
    \[
    \vec{\mathbb{P}}^{\ssup f}(X_{t+h} = j \mid \mathcal{F}_t)  =  f(k_{i,j}(\X, t)) h + o(h),
    \]
    where \( \mathcal{F}_n \) is the sigma-algebra generated by the process up to time \( n \).
Set 
$$\Theta(f, k) := \prod_{(i,j) \in \vec{E}} \prod_{s=1}^{k_{i,j}} f(s) \qquad  \mbox{ and } \qquad
\Lambda(f, n, i) := \prod_{\substack{j = 0 \\ j \ne i}}^{n-1} \frac{1}{f(j) - f(i)}.
$$
\begin{mdframed}[style=MyFrame2]
\begin{theorem}\label{th:firstfo0} Let $X$ be dERRW on G. Fix  a connected subgraph $G' \subset G$, where $G' = (V', E')$ and $0 \in V'$. Fix $i_1 \in V'$ and a local time profile $\ell \in (0,\infty)^{V'}$ with $\sum_i \ell_i = t$ and $k \in \N^{E'}$ such that  $b(k)_i = \delta_{0}(i) - \delta_{i_1}(i)$ for all $i \in V'$.  Finally let $\vec{T}$ be an oriented  spanning tree of $G'$ rooted at $i_1$. One has 
\begin{equation}\label{eq:loctim}
\begin{aligned}
&\vec{\mathbb{P}}^{\ssup f} \big(k(\X, t) = k,\; \ell(\X, t) \in (\ell, \ell+ d\ell),\; \vec{T}(\X, t) = \vec{T} \big) \nonumber \\
&=  \Theta(f, k)  \Big(\prod_{ij \in \vec{E}'}   \sum_{i=1}^{k_{i,j}}\Lambda(f, k_{i,j}+ h_{i,j}(\vec{T}), i) ( {\rm e}^{-f(k_{i, j}) \ell_i} - (1- h_{i,j}(\vec{T})){\rm e}^{-f(k_{i,j}+1) \ell_i }) \Big) m_t(d \ell). 
\end{aligned}
\end{equation}
\end{theorem}
\end{mdframed}

\subsubsection{Directed once-reinforced random walk} The directed once-reinforced random walk with parameter $a > 0$, abbreviated  dORRW($a$) corresponds to the case $f(1)=a$ and $f(j) = 1$ for all  $j \ge 2$.  We use $\vec{\P}^{\ssup a}$ the probability measure associated to this process.
Let \( b \in \mathcal{I} \) be such that \( b_i = \delta_{i_0}(i) - \delta_{i_1}(i) \) for all \( i \in V \). For each pair of neighbors $\{i, j\}$, choose a unique orientation  such that \( b_{i,j} \ge 0 \). Let $\vec{E}^+$ be the collection of these oriented edges. In particular, \( \vec{E}^+ = \{ (i, j)\in \vec{E} \colon b_{i, j} \ge 0 \} \).
Define 
$$
J_{v, w} (z) = \sum_{k=0}^\infty \frac 1{\Gamma(k+v+1)\Gamma(k+w+1)} \left(\frac z2 \right)^{2k + v+w}.
$$
Note that 
$ J_{v, 0}(z) =  I_v(z), $ where $I_v$ is the well-known modified Bessel function of the first kind.
\vspace{.5cm}
\begin{mdframed}[style=MyFrame2]
\begin{theorem}\label{eq:thD}
Let $(D_{i,j})_{(i, j) \in \vec{E}}$ be collection of i.i.d.  geometric random variables with probability mass function 
$$
\vec{\P}^{\ssup a} (D_{i,j} = n) = a (1-a)^n, \qquad \mbox{for $n \ge 0$}.
$$
Fix  a connected subgraph $G' \subset G$, where $G' = (V', E')$ and $0 \in V'$. Fix $i_1 \in V'$ and a local time profile $\ell \in (0,\infty)^{V'}$ with $\sum_i \ell_i = t$ and $k \in \N^{E'}$ such that  $b(k)_i = \delta_{0}(i) - \delta_{i_1}(i)$ for all $i \in V'$.  Finally let $\vec{T}$ be an oriented  spanning tree of $G'$ rooted at $i_1$. One has 
\begin{equation}\label{eq:loctim0}
\begin{aligned}
&\vec{\mathbb{P}}^{\ssup a} \big(\tilde{b}(\X, \sigma) = \tilde{b},\; \ell(\X, \sigma) \in (\ell, \ell+ d\ell),\; \vec{T}(\X, \sigma) = \vec{T} \big)  \\
&=\frac{a^{|\vec{E}|}}{(1 - a)^{\|k\| -|\vec{T}|}} e^{-a \sum_{i \in V'} deg_i \times  \ell_i } \prod_{(i, j) \in \vec{E}' }   \frac{\gamma(k_{i,j} - h_{i,j}(\vec{T}), (1 - a)\ell_i)}{\Gamma(k_{i,j} -h_{i,j}(\vec{T}))}  m_\sigma(d \ell)\\
&= \E^{\otimes D_{i,j}}\left[{\rm e}^{-  \sum_{i \in V'} {\rm deg}_i \ell_i} \prod_{\{i,j\} \in E'} J_{D_{i,j}+|\widetilde{b}_{i,j}|,  D_{j, i}}(2 \sqrt{\ell_i \ell_j})\left(\prod_{i \in V}  \ell_i^{\frac{{\rm Div} (D)}2 +\widetilde{b}_i}\right)   m_\sigma(d \ell)\right],
\end{aligned}
\end{equation}
where ${\rm deg}_i$ is the degree of $i$ in $G$. Moreover,
\begin{equation}\label{eq:loctim0.01}
\begin{aligned}
&\vec{\mathbb{P}}^{\ssup a} \big(\ell(\X, \sigma) \in (\ell, \ell+ d\ell),\; \vec{T}(\X, \sigma) = \vec{T}\big)  \\
&=  \frac 1{(2 \pi)^{|V|-1}}  \int_{[0, 1]^{V\setminus{i_1}}} \vartheta(x)   \prod_{(s, j) \in \vec{E}'\setminus{\vec{T}}} \Big[ (1-a)({\rm e}^{2\pi i (x_s - x_j)}  -  1) + a {\rm e}^{2\pi i (x_s - x_j)}  e^{\ell_s ({\rm e}^{2\pi i (x_s - x_j)} - (1-a))} \Big]\\
&   \qquad \qquad \qquad \prod_{(s, j) \in \vec{T}} {\rm e}^{2\pi i (x_s - x_j)}  \left[ 2 {\rm e}^{2\pi i (x_s - x_j)}  - (1-a) + {\rm e}^{2\pi i (x_s - x_j)}  {\rm e}^{\ell_s ({\rm e}^{2\pi i (x_s - x_j)}  - (1-a))} \right]\prod_{i=1}^{|V|-1} {\rm d} x_i,
\end{aligned}
\end{equation}
where $x_{i_1} =0$ and 
$$
\vartheta(x) := \prod_{(i,j) \in \vec{E}}\frac{e^{-a\ell_{s}}}{{\rm e}^{2\pi i (x_s - x_j)}  - (1-a)}.
$$
\end{theorem}
\end{mdframed}
\vspace{.4cm}
Large deviations theory provides a framework for quantifying the probabilities of rare events and atypical fluctuations in stochastic processes. For a simple random walk on a finite graph \( G = (V, E) \), the Large Deviations Principle (LDP) characterizes the exponential decay rate of the probability that the empirical measure or empirical flow of the walk deviates from its typical  behavior. Define the so-called Donsker--Varadhan rate function
\begin{equation}\label{eq:ldp}
{\rm Dir}(x) := \begin{cases}  \sum_{i \in V} \sum_{j \sim i} (\sqrt{x_i} - \sqrt{x_j})^2\qquad \mbox{ if $x \in (0, 1)^V$}\\
\infty \qquad \mbox{otherwise}.
\end{cases}
\end{equation}
 the empirical measure of the simple random  walk satisfies, for all Borel set $A$ of the $|V|-1$-simplex,  
\begin{equation}
\begin{aligned}
\liminf_{t \to \infty} \frac 1t \log \P(\ell(\X, t) \in A) \ge - \inf_{\ell \in A^o} {\rm Dir}(\ell) \\
\limsup_{t \to \infty} \frac 1t \log \P(\ell(\X, t) \in A) \le - \inf_{\ell \in \overline{A}} {\rm Dir}(\ell).
\end{aligned}
\end{equation}

\begin{mdframed}[style=MyFrame2]
\begin{theorem}\label{eq:LDP}
Fix $a \in (0, 1]$. Let $G = (V, \vec{E})$ be a finite graph and fix $a>0$.  Let $A$ be any measurable set of the $|V|-1$-simplex with $A \subset (0, \infty)^V$. One has
\begin{equation}
\begin{aligned}
&\liminf_{t \to \infty} \frac 1t \log \vec{\P}^{\ssup a}(\ell(\X, t) \in A) \ge - \inf_{\ell \in A^o} {\rm Dir}(\ell)\\
&\limsup_{t \to \infty} \frac 1t \log \vec{\P}^{\ssup a}(\ell(\X, t) \in A) \le (1-a) \sup_{\ell \in \overline{A}}\sum_{i} {\rm deg}_i \frac{\ell_i}t
- \inf_{\ell \in \overline{A}} {\rm Dir}(\ell)
\end{aligned}
\end{equation}
\end{theorem}
\end{mdframed}

Consider now dORRW $\X$ on the directed graph with vertex set $\Z^d = (V_d, \vec{E}_d)$, with $d \ge 2$, where $(x, y) \in \vec{E}_d$ if and only if $x$ and $y$ differ by one coordinate and the magnitude of the difference is one.
Define $\vec{\mathfrak{G}}_t := (\vec{\mathfrak{R}}_t,  \vec{\mathfrak{C}}_t)$, where $\vec{\mathfrak{R}}_t = \{v \in \Z^d \colon \exists u \in [0, t] \mbox{ such that } X_u = v\}$, while 
$$
\vec{\mathfrak{C}}_t := \{(i, j) \in \Vec{E}_d \colon  \exists u \in [0, t] \mbox{ such that } X_{u^-} = i \mbox{ and } X_u = j\}.
$$
Our next result shows that the shape theorem cannot hold for  ORRW on  the oriented grid $\Z^d$ unless the boundary of the range evolves in a very irregular way. we consider ORRW $\X$ on oriented  $\Z^d$, with $d \ge 2$. In particular,  the first time an oriented edge $(x,y)$ is traversed, its weight is reinforced, while $(y,x)$ is not, unless already traversed in the past. Set $\vec{\mathfrak{G}}_t = (\mathfrak{R}_{t}, \vec{ \mathfrak{C}}_{t})$ where  $\mathfrak{R}_{t}$ and $\vec{ \mathfrak{C}}_{t}$ are the vertex range and the (oriented) edge range, respectively, of $\X$ by time $t$.  Denote by $(\tau_i)_i$ the jump times of this process.~Fix a sequence $u \colon \N \rightarrow (0, \infty)$ such that $u_N= o(N^{\theta})$, for some $\theta \in (0, 1/d)$.   Let $e_1$ be the unit vector $(1, 0, \ldots, 0) \in \Z^d$. Let $u_N = o(N^{\theta})$ for some $\theta<1/d$.
\begin{mdframed}[style=MyFrame2]
\begin{theorem}\label{th:main4} Consider dORRW $X$ on $\Z^d$. Fix a sequence $u \colon \N \rightarrow (0, \infty)$, such that $u_N = o(N^{1/d})$. Let 
 $ \mathcal{A}_N := \{{\rm Ball}(0, (1-\eps)u_N) \subset \vec{\mathfrak{G}}_{\tau_N} \subset {\rm Ball}(0, (1+\eps)u_N)\},$ where these are graph inclusions. 
Define 
\begin{eqnarray*}
\mathcal{A} &:=& \bigcup_{m>0} \bigcap_{n > m} \mathcal{A}_n\\
\mathcal{B} &:=& \Big\{\lim_{N \ti} N^{-\frac 12} \sum_{v \in \partial \mathfrak{R}_{\tau_N}} \1_{(v, v - e_1) \notin \vec{ \mathfrak{C}}_{\tau_N}} -\1_{(v, v + e_1) \notin \vec{ \mathfrak{C}}_{\tau_N} } =0\Big\}
\end{eqnarray*}
One has $  \P^{\ssup a}\big(\mathcal{A} \cap \mathcal{B}\big) =0.$
\end{theorem}
\end{mdframed}


\section{Once-Reinforced random walks}
Our first task is to introduce a new change of measure formula.
 \begin{proof}[Proof of Proposition~\ref{pr:changem}]
For any $N \in \N$ let $[N] := \{1, 2, . . . ,N\}$. Fix an edge path $\mathbf{e}_n = (e_1, e_2, \ldots e_{n})$, such that each pair of consecutive edges in the path share an endpoint, consistently with a vertex path. Denote by $m$ the number of distinct edges in  $\mathbf{e}_n$, and by $(0, x_1, x_2, \ldots, x_n)$ the vector of endpoints of the 
edge path, such that $x_i$ is a neighbour of $x_{i-1}$ for all $i \in [n]$. For all $k \in \{0, 1,2 \ldots n-1\}$, let $\mathfrak{b}_k$ be the number of edges incident to $x_k$ that are not listed in $(e_1, e_2, \ldots, e_{k-1})$, i.e. have not been traversed before time $k$. Set $\mathfrak{z}_k =1$  if $e_k$  coincides with one of the coordinates of $(e_1, e_2, \ldots, e_{k-1})$, and  $\mathfrak{z}_k =0$ otherwise. Let $t_0 < t_1 <\ldots< t_{n}$. Notice that given that the process is at $x_{k}$ at time $t_{k}$, and given that the edge traversed  in the past are $(e_1, e_2, \ldots, e_{k-1})$, the likelihood  that the first jump happens in the interval  $(t_{k+1}, t_k + {\rm d}t_{k+1})$ is 
$$
(a \mathfrak{b}_k  + (2d - \mathfrak{b}_k) ) \exp\{-( a \mathfrak{b}_k  + (2d - \mathfrak{b}_k) ) (t_k - t_{k-1})\} {\rm d}t_{k+1}.
$$
Moreover,  the conditional probability that the next jump is towards  $x_{k+1}$  is
$
\frac {a(1-\mathfrak{z}_k) +  \mathfrak{z}_k}{a \mathfrak{b}_k  + (2d - \mathfrak{b}_k)}.
$ Let $B:= \bigcap_{k=1}^n\{ X_{\tau_k} =x_k,  \tau_k \in (t_k, t_k + {\rm d}t_k)\}$.
 Hence,
\begin{equation}
\begin{aligned}\label{eq:pivotaldecomp}
\P^{\ssup a}\left(B\right)&= \prod_{k=1}^{n} \frac {a(1-\mathfrak{z}_k) +  \mathfrak{z}_k}{a \mathfrak{b}_k  + (2d - \mathfrak{b}_k)} (a \mathfrak{b}_k  + (2d - \mathfrak{b}_k) ) \exp\{-( a \mathfrak{b}_k  + (2d - \mathfrak{b}_k) ) (t_k - t_{k-1})\}\prod_{k=1}^n {\rm d}t_k\\
 &= a^m \exp\left\{(1-a)\sum_{k=1}^{n} \mathfrak{b}_k (t_k - t_{k-1})\right\} \exp\{ -2d t_{n} \}\prod_{k=1}^n {\rm d}t_k.
 \end{aligned}
\end{equation}
It is easy to recognise that the last factor in the right-hand side of \eqref{eq:pivotaldecomp} coincides with the simple random walk measure. Hence 
$$
\P^{\ssup a}\left(B\right)  = a^m \exp\left\{(1-a)\sum_{k=1}^{n} \mathfrak{b}_k (t_k - t_{k-1})\right\} \P\left(B\right).
 $$
\end{proof}
Set  $H_n := \inf\{t \ge 0: |\mathfrak{C}_t| = n\}$, and set 
$\Delta_n :=
\sum_{
e\in E}
\big(T_e(H_{n+1}) - T_e(H_n)\big).$
 In this section
we prove that $(\Delta_n)_n$ are  i.i.d. exponential($a$), under $\P^{\ssup a}$.
\begin{mdframed}
[style=MyFrame1]
\begin{proposition}\label{pr:changem1}
$
\P(\Delta_n > t\;|\; \mathcal{F}_{H_n}) = {\rm e}^{-t}.
$
\end{proposition}
\end{mdframed}
\begin{proof}
Let $A = (V_A, E_A)$ be a connected sub-graph of $\Z^d$, with exactly $n$ edges.  Denote by ${\rm deg}_A(x)$ the degree of $x \in A$ induced in  $A$. Moreover, $y \sim_A x$
means that both vertices are in $A$ and they are joined by an edge in this graph.  We prove that for any  $B \in \mathcal{F}_{H_n}$, such that $B \cap \{\mathfrak{C}_{H_n} = E_A\} \neq \emptyset$, one has
$$\P(\Delta_n > t\,|\, \mathfrak{C}_{H_n} = E_A, B) = {\rm e}^{-t}.$$
To avoid confusion with the derivative sign, in this proof only, we use ${\rm dim}$ for the dimension
of the lattice. Set, for $x \in V_A$,
$$
f_x(t) = \P(\Delta_n > t| X_{H_n} = x, \mathfrak{C}_{H_n} = E_A, B).
$$
We next prove that the collection of functions $(f_x(\cdot))_{x\in V_A}$ satisfy the system of differential equations
\begin{equation}\label{eq:diffeq}
\left(
2 \cdot {\rm dim} - \rm{deg}_A(x) \right)
\frac {\rm d}
{{\rm d}t}
f_x(t) = -2 \cdot {\rm dim} \cdot f_x(t) +\sum_{y \colon y \sim_A x} 
f_y(t),
\end{equation}
with initial conditions $f_x(0) = 1$ for all $x \in V_A$. Notice that $f_x(t) = e^{-t}$ is the unique solution of \eqref{eq:diffeq}. In fact, as $A$  is finite, then \eqref{eq:diffeq} is a finite system of differential equations.
Next, we turn to the proof of \eqref{eq:diffeq}. Set $q_x = 2 \cdot {\rm dim} - {\rm deg}_A(x)$. The reader can think of $A$ as the outcome of  $(\mathfrak{R}_{H_n}, \mathfrak{C}_{H_n})$. Hence, 
$q_x$ can be thought of the number of edges incident to $x$ that have not yet been traversed. Next, we
argue that
$$
\begin{aligned}
f_x(t + q_x h) &= \P(\Delta_n > t+ q_x h\;\big{|}\; X_{H_n} = x, \mathfrak{C}_{H_n} = E_A, B)\\
&= \P(\Delta_n > t \;\big{|}\; X_{H_n} = x, \mathfrak{C}_{H_n} = E_A, B)(1 - 2\cdot{\rm dim} \cdot h)\\
&\qquad + h
\sum_{y \colon y \sim_A x} 
\P(\Delta_n > t \;\big{|}\; X_{H_n} = y, \mathfrak{C}_{H_n} = E_A, B) + o(h)\\
&= f_x(t)(1 -2h \cdot {\rm dim} ) + h
\cdot
\sum_{y \colon y \sim_A x} 
f_y(t) + o(h).
\end{aligned}•
$$
The second equality is obtained using the following reasoning:\\
\noindent $\bullet$ If the random walk stays at $x$  for the initial $h$ unit of times it contributes $q_xh$ to $\Delta_n$. The
latter holds with probability $(1 -2 \cdot {\rm dim} \cdot h) + o(h)$, whereas\\
\noindent $\bullet$  if the process jumps to a given neighbor $y \sim_A x$, which holds with probability $h$, then it is very
unlikely that it does jump again in the remaining $h$ unit of times, running the time only the time when the process is exposed to untraversed edges.
\end{proof}
\begin{mdframed}
[style=MyFrame1]
\begin{corollary}\label{co:coro1} The random variables  $(\Delta_n)_n$ are i.i.d. exponential(1) under the measure $\P$.
\end{corollary}
\end{mdframed}
\begin{mdframed}
[style=MyFrame1]
\begin{proposition}\label{pr:TunPa} The random variables  $(\Delta_n)_n$ are i.i.d. exponential($a$) under the measure $\P^{\ssup a}$.
\end{proposition}
\end{mdframed}
\begin{proof} Using $|\mathfrak{C}_{H_n}| =n$, and the change of measure formula from Proposition~\ref{pr:changem}, one has 
$$
\begin{aligned}
&\P^{\ssup a}\big(\Delta_1\in (x_1, x_1+{\rm d} x_1), \ldots, \Delta_n \in (x_n, x_n +{\rm d} x_n)\big)\\
&= \E\left[\exp\left\{(1-a)\sum_{e \in E} T_e(H_n) + (\log a) n \right\}\1_{\Delta_1\in (x_1, x_1+{\rm d} x_1), \ldots, \Delta_n \in (x_n, x_n +{\rm d} x_n)}\right] \\
&= a^n {\rm e}^{(1-a) \sum_{i=1}^n x_i} \P\big(\Delta_1\in (x_1, x_1+{\rm d} x_1), \ldots, \Delta_n \in (x_n, x_n +{\rm d} x_n)\big) = a^n {\rm e}^{-a \sum_{i=1}^n x_i},
\end{aligned}
$$
where in the last step we used Corollary~\ref{co:coro1}.
\end{proof}
Recall that $(\tau_n)_n$ is the sequence of times when the process jumps.
\begin{proof}[Proof of Theorem~\ref{th:main}]
On the event $\{|\mathfrak{R}_{\tau_N}| \le uN^{d/(d+2)}\}$, one has that $|\mathfrak{C}_{\tau_N}| \le d uN^{d/(d+2)}$.  Set $J_N :=  d uN^{d/(d+2)}$. For the case $a \in (0,1)$, one has
$$
\begin{aligned}
\P^{\ssup a} \left(|\mathfrak{R}_{\tau_N}| \le uN^{d/(d+2)}\right) &= \E\left[ {\rm e}^{(1-a)\sum_e T_e(\tau_N) +(\log a)|\mathfrak{C}_{\tau_N}|} \1_{|\mathfrak{R}_{\tau_N}| \le uN^{d/(d+2)}}\right]\\
&\le \E\left[ {\rm e}^{(1-a)\sum_e T_e(H_{J_N}) } \1_{|\mathfrak{R}_{\tau_N}| \le uN^{d/(d+2)}}\right] \qquad\mbox{\color{blue}(As $\tau_n \le H_{J_N}$, and $a \in (0,1)$)}.
\end{aligned} 
$$
Using H\"older inequality, we obtain
$$
\begin{aligned}
\P^{\ssup a} \left(|\mathfrak{R}_{\tau_N}| \le uN^{d/(d+2)}\right)&\le  \E\left[ {\rm e}^{p(1-a)\sum_e T_e(H_{J_N})}\right]^{1/p} \P\left(|\mathfrak{R}_{\tau_N}| \le uN^{d/(d+2)}\right)^{\frac{p-1}p}\\
&= \E\left[ {\rm e}^{p(1-a)\sum_e T_e(H_{J_N})}\right]^{1/p} \P\left({\rm e}^{-|\mathfrak{R}_{\tau_N}|} \ge {\rm e}^{-uN^{d/(d+2)}}\right)^{\frac{p-1}p}\\
&\le   \left(\frac 1{1 - p(1-a)}\right)^{2duN^{d/(d+2)}/p} {\rm e}^{-((\psi_d -u)\frac{p-1}p + o(1)) N^{d/(d+2)}},
\end{aligned}
$$
where in the last step we used the moment generating function for exponential(1) and Theorem \ref{th:DonVar}  (see Appendix) due to Donsker and Varadhan.\\
 For $a \ge 1$, one has 
$$
\begin{aligned}
\P^{\ssup a} \left(|\mathfrak{R}_{\tau_N}| \le uN^{d/(d+2)}\right) &= \E\left[ {\rm e}^{(1-a)\sum_e T_e(\tau_N) +(\log a)|\mathfrak{C}_{\tau_N}|} \1_{|\mathfrak{R}_{\tau_N}| \le uN^{d/(d+2)}}\right]\\
&\le a^{ud N^{d/(d+2)}} \P\left(|\mathfrak{R}_{\tau_N}| \le uN^{d/(d+2)}\right)\\
&\le a^{ud N^{d/(d+2)}}   {\rm e}^{-((\psi_d -u) + o(1)) N^{d/(d+2)}}.
\end{aligned}
$$
Finally, using the definition of $\nu$ given in \eqref{eq:defnu} we get the advertised bound. Notice that  the quantity $o(1)$ appearing in the last expression does not depend on $a$.
\end{proof}

\subsection{Proof of Theorem~\ref{th:maintamen} }
Suppose that  $\mathcal{G}$ is a nonamenable graph rooted at $0$ with  $X_0=0$,  then 
\begin{equation}\label{eq:non-amen-Che}
 \P(X_{\tau_n} =   0) \le {\rm e}^{-\kappa  n}, \qquad \mbox{for all $n\in \N$,} 
 \end{equation}
where $\kappa>0$ depends only on the Cheeger's constant $\gamma_{\mathcal{G}}$, and $\P$ stands for simple random walk measure (see, e.g., Theorem 14 in  \cite{ACK14} or Theorem 6.7 page 184 in \cite{lyons-peres}). Recall that $H_n$ is the first time $t$ such that 
$|\mathfrak{C}_t| = n$.
 \begin{proof}[Proof of Theorem~\ref{th:maintamen} a)]
For $p>1$ with $p(1-a)<1$, one has
$$
\begin{aligned}
\P^{\ssup a}(X_{\tau_n} =   0) &= \E[{\rm e}^{(1-a)\sum_e T_e(\tau_n) + (\log a)|\mathfrak{C}_{\tau_n}|}\1_{X_{\tau_n} =   0}]\\
&\le  \E[{\rm e}^{(1-a)\sum_e T_e(H_n)}\1_{X_{\tau_n} =   0}] \qquad \mbox{\color{blue}(as $\tau_n \le H_n$ and $\log a<0$)}\\
&\le \E[{\rm e}^{(1-a)p\sum_e T_e(H_n)}]^{\frac 1p} \P(X_{\tau_n} =   0)^{\frac{p-1}p}  \qquad \mbox{\color{blue}(Hölder's inequality)}\\
&= \left(\frac{1}{1 -p(1-a)}\right)^{\frac np} {\rm e}^{-\frac{p-1}p\kappa n}.
\end{aligned}
$$
In the last equality, we used that $\sum_e T_e(H_n)$ is Gamma$(1, n)$ under the measure $\P$ (see Proposition~\ref{pr:TunPa}).
For all $a$ close enough to $1$,  the right-hand side is bounded by $\gamma^n $ for some $\gamma \in (0,1)$.
\end{proof}
\begin{proof}[Proof of Theorem~\ref{th:maintamen} b)] Fix $\eps$ small enough to be specified below. Set 
\begin{equation}\label{eq:nonam1}
\begin{aligned}
\P^{\ssup a}(X_{\tau_n} =   0, |\mathfrak{C}_{\tau_n}| < \eps n )
&= \E[{\rm e}^{(1-a)\sum_e T_e(\tau_n) + (\log a)|\mathfrak{C}_{\tau_n}|}\1_{X_{\tau_n} =   0,\, |\mathfrak{C}_{\tau_n}| < \eps n}] \\
&\le \E[{\rm e}^{(1-a)\sum_e T_e(H_{\ceil{\eps n}}) }\1_{X_{\tau_n} =   0}] \qquad \mbox{\color{blue}(as $\tau_n \le H_{\ceil{\eps n}}$)}
\end{aligned}
\end{equation}
Using the fact that $\sum_e T_e(H_{j})$ is Gamma$(j, a)$,  combined with Hölder inequality, and $p>1$ with $(1-a)p<1$, one has
\begin{equation}\label{eq:expb}
\begin{aligned}
 \P^{\ssup a}(X_{\tau_n} =   0, |\mathfrak{C}_{\tau_n}| < \eps n ) 
 &\le \left(\frac{1}{1 -p(1-a)}\right)^{\frac {\eps n+1} p} \P(X_{\tau_n} =   0)^{\frac{p-1}p} \le \left(\frac{1}{1 -p(1-a)}\right)^{\frac {\eps n+1} p} {\rm e}^{-\frac{p-1}p\kappa n}.
 \end{aligned}
\end{equation}
Choose $\eps$ small enough such that the previous expression is summable in $n$.
For this choice of $\eps$, one has that 
\begin{equation}\label{eq:BCNA}
\sum_{n=1}^\infty \P^{\ssup a}(X_{\tau_n} =   0, |\mathfrak{C}_{\tau_n}| < \eps n ) < \infty.
\end{equation}
Notice that $\tau_n > 2n$ holds only for finitely many $n$, under $\P^{\ssup a}$.   Hence $\lim_{t\ti} \frac{|\mathfrak{C}_t|}{t} = 0$ implies that $\lim_{n\ti} \frac{|\mathfrak{C}_{\tau_n}|}{n} = 0$. 
Finally, notice that  $\{\X$ is recurrent$\} \cap \{\lim_{t\ti} \frac{|\mathfrak{C}_t|}{t} = 0\} \subset B$ where $B$ is the event that there are infinitely many times $\tau_n$ such that $X_{\tau_n} =   0, |\mathfrak{C}_{\tau_n}| < \eps n$. 
 Using \eqref{eq:BCNA}, combined with the first Borel-Cantelli Lemma, we obtain $\P^{\ssup a}(B)=0$. The latter implies that 
 $$
 \P^{\ssup a}(\{\X \mbox{ is recurrent}\} \cap \{\lim_{t\ti} \frac{|\mathfrak{C}_t|}{t} = 0\}) = 0.
 $$
 \end{proof}
\subsection{Proof of Theorem~\ref{th:main8}}
Fix a graph $G= (V, E)$ where simple random walk is transient.~Let $S$ be the first return time to the origin.
Assume that $\limsup_{a \uparrow 1} \E^{(a)}[S]<\infty \newconstant{co:new00}$ and reason by contradiction.
  Let $\useconstant{co:new00} :=  - \log \P(S  <\infty)$. As simple random walk is transient, then $\useconstant{co:new00} \in (0, \infty)$.  Fix $\alpha>1$ such that $\log \alpha < c_1$. Define the probability measure
\[
\widetilde{\P}  :=  {\rm e}^{\useconstant{co:new00} } \mathbf{1}_{\{S <\infty\}} \P.  \newconstant{co:new2}
\]
Let $\useconstant{co:new2} := \floor{(1-1/\alpha)^{-1}\E^{(a)}[S]}+1 $. Set 
$W := \P^{\ssup a} (S   < \useconstant{co:new2} )$. We show next that $ W \ge 1/2$. In fact, as $S$ takes integer values,
$$
\P^{\ssup a} (S  \ge  \useconstant{co:new2} )  =  \P^{\ssup a} (S  \ge  2\E^{(a)}[S])\le \frac{\E^{\ssup a} [S]}{(1-1/\alpha)^{-1} \E^{(a)}[S]} = (1-1/\alpha).
$$
Define
\[
\mathbb{Q}^{(a)} := W^{-1} {\rm e}^{(1-a)\sum_{e}T_e(\tau_{\useconstant{co:new2}})+(\log a)|\mathfrak{C}_{\tau_{\useconstant{co:new2} }}|}\mathbf{1}_{S<\useconstant{co:new2} } \P.
\]

Consider the relative entropy, also known as Kullback-Leibler divergence
\begin{equation}
\mathrm{KL}(\widetilde{\P} \| \mathbb{Q}^{(a)}) := \int \log\left(\frac{d\mathbb{Q}^{(a)}}{d\widetilde{\P}}\right) d\mathbb{Q}^{(a)}.
\end{equation}
Using $ W \ge 1/\alpha$, combined with  $\sum_{e \in E} T_e(\tau_{\useconstant{co:new2} }) \le \sum_{e \in E} T_e(H_{\useconstant{co:new2}  })$ and  $a\in (0,1)$, one has
\begin{equation}
\begin{aligned}
\mathrm{KL}(\widetilde{\P} \| \mathbb{Q}^{(a)})  &\le   - \useconstant{co:new00}  + (1 - a) \mathbb{E}^{(a)}\left[ \sum_{e \in E} T_e(H_{\useconstant{co:new2} })  \right] + \log \alpha\\
&\le  - \useconstant{co:new00} + 4 \frac{1 - a}a\E^{(a)}[S] + \log \alpha, \qquad \qquad \mbox{\color{blue}{(as $T_e(H_{\useconstant{co:new2} })\sim$ Gamma($\useconstant{co:new2}$, $a$))}}. 
\end{aligned}
\end{equation}
 If we choose $a<1$ close enough to 1 we obtain that $\mathrm{KL}(\widetilde{\P} \| \mathbb{Q}^{(a)})<0$ for all large enough $N$.
 This produces a contradiction as the entropy cannot be negative.
 \hfill $\qed$

\section{ Local time theorems for general self-interacting random walks} 
\begin{proof}[Proof of Theorem~\ref{th:supermainth}]
Recall that  $V' \subset V$ is the support of $k$, i.e. $i\in V'$ if there exists $j$ such that  $\max\{k_{i, j}, k_{j, i}\} \ge 1$. Fix an environment $\omega$. The event $\{k(X, t) = k,\; \ell(X, t) \in (\ell, \ell+ d\ell),\; \vec{T}(X, t) = \vec{T}\}$ holds if and only if  the following holds for each vertex $i \in V'$.
\begin{itemize}[ leftmargin=*, itemsep=0pt, parsep=0pt, topsep=0pt, partopsep=0pt]
\item For each neighbor $j$ of $i$, such that $(i, j) \notin \vec{T}$, there are exactly $k_{i, j}$   jumps from $i$ to $j$ in the time interval $(0, \ell_i)$. The probability that these are exactly $k_{i, j}$ is 
$
\mathcal{P}(f_{i,j}, k_{i,j}, \ell_i).
$
\item For the neighbor  $j$ of $i$, where $i \neq i_1$, such that $(i, j) \in \vec{T}$, we require exactly one jump from $i$ to $j$  in  interval $(\ell_i, \ell_i + {\rm d} \ell_i)$, which holds with probability   $
\mathcal{P}^*(f_{i,j}, k_{i,j}, \ell_i) {\rm d} \ell_i.
$
\end{itemize}
As for \eqref{eq:loctim-1.01} it is a direct application of the Poisson summation formula on the lattice, with the constraint given by the divergence-free flows.
\end{proof}
\subsection{Proof of Theorem~\ref{th:RWREdirec}: RWRE}
\begin{proof}[Proof of Theorem~\ref{th:RWREdirec}]
It is immediate to check that $\X$ satisfies the conditions in Theorem\ref{th:supermainth} under the quenched measure, and that
\begin{eqnarray*}
\mathcal{P}(f_{i,j}, k_{i,j}, \ell_i)  &=& \frac{(\omega(i, v) \ell_i)^{k_{i,v}}}{(k_{i,v})!} {\rm e}^{-\omega(i, v) \ell_i}\\
\mathcal{P}^*(f_{i,j}, k_{i,j}, \ell_i) &=&
\frac{(\omega(i, j) \ell_i)^{k_{i,j} - 1}}{(k_{i,j} - 1)!} {\rm e}^{-\omega(i, j) \ell_i} \omega(i, j).
\end{eqnarray*}
Hence
\begin{equation}\label{eq:loctim11}
\mathbf{P}_\omega \big(k(X, t) = k,\; \ell(X, t) \in (\ell, \ell+ d\ell),\; \vec{T}(X, t) = \vec{T} \big)= \prod_{(i,j) \in \vec{E}} \frac{(\omega(i, j))^{k_{i,j}} \ell_i^{k_{i,j} - h_{i,j}(\vec{T})}}{(k_{i,j} - h_{i,j}(\vec{T}))!} {\rm e}^{-\omega(i, j) \ell_i} m_t({\rm d} \ell).
\end{equation}
Observe that $\prod_{ j\colon (i,j) \in \vec{E}}  {\rm e}^{-\omega(i, j) \ell_i} = {\rm e}^{-\ell_i}$ as $\sum_{j \colon j \sim i} \omega(i, j) =1$. Hence $\prod_{(i,j) \in \vec{E}} {\rm e}^{-\omega(i, j) \ell_i} = {\rm e}^{-t}$. Moreover, observe that $\sum_{j \colon j \sim i} k_{i,j}  - h_{i,j}(\vec{T}) = k_i - (1- \delta_{i_1})$ as there is exactly one exit edge from $i\neq i_1$ and none from $i_1$. Hence the right-hand side of \eqref{eq:loctim11} equals to
\begin{equation}
\begin{aligned}
&= {\rm e}^{-t} \left(\prod_{(i,j) \in \vec{E}} \frac{\ell_i^{k_{i,j} - h_{i,j}(\vec{T})}}{(k_{i,j} - h_{i,j}(\vec{T}))!} \right) \left(\prod_{(i,j) \in \vec{E}} \omega(i, j)^{k_{i,j} }\right)m_t({\rm d} \ell) \\
&= {\rm e}^{-t} \ell_{i_1} \left(\prod_{i \in V} \ell_i^{k_i -1}\right) \left(\prod_{(i,j) \in \vec{E}} \frac{1}{(k_{i,j} - h_{i,j}(\vec{T}))!} \right)\left(\prod_{(i,j) \in \vec{E}} \omega(i, j)^{k_{i,j} }\right)m_t({\rm d} \ell).
\end{aligned}
\end{equation}
By averaging over the  environment, we obtain 
$$
\E\left[\prod_{(i,j) \in \vec{E}} \omega(i, j)^{k_{i,j} }\right] = \prod_{i \in V'}  \Psi((k_{i,j} )_{j \sim i}).
$$
\end{proof}
\begin{proof}[Proof of Theorem~\ref{th:DRW}]
We compute $\Psi$ for the case where $\omega(i)$ are i.i.d. Dirichlet$(1, 1, \ldots, 1)$. Denote by $i_1, i_2, \ldots i_{\Delta}$ the neighbors of $i$, with a specific order. 
$$
\begin{aligned}
\E[\prod_{j \colon j \sim i} \omega(i, j)^{k_{i,j} }] &= \Gamma(\Delta) \int_{\mathcal{S}(1, \Delta)} \prod_{r=1}^{\Delta} x_r^{k_{i, i_r} } m({\rm d} x) =\frac{\Gamma(\Delta)}{\Gamma(k_{i} +\Delta)}  \prod_{j \colon j \sim i} \Gamma(k_{i,j}  +1). 
\end{aligned}
$$
Hence, plugging the previous expression in \eqref{eq:loctim}, we obtain
\begin{equation}\label{eq:loctim101}
\begin{aligned}
\mathsf{P} \big(&k(X, t) = k,\; \ell(X, t) \in (\ell, \ell+ d\ell),\; \vec{T}(X, t) = \vec{T} \big) \\
&= {\rm e}^{-t} \Gamma(\Delta)^{|V'|}  \left(\prod_{i \in V'} \frac{\ell_i^{k_i -1}}{\Gamma(k_{i} +\Delta)}\right) \ell_{i_1} \left(\prod_{(i,j) \in \vec{T}} k_{i,j}\right) m_t({\rm d}\ell)\\
\end{aligned}
\end{equation}
By summing over the elements $\vec{T} \in \mathcal{T}_{i_1, V'}$, one obtains  the result.
\end{proof}
By integrating over the simplex,  one obtains.
\begin{equation}\label{eq:dir1}
\begin{aligned}
&\mathsf{P}^{\ssup{Dir}} \big(k(X, t) = k,\; X_t = i_1\big)= \sum_{\vec{T} \in \mathcal{T}_{i_1, V'}}  \Gamma(\Delta)^{|V'|} \frac{t^{\|k\|}}{\|k\|!} {\rm e}^{-t} \Big(\prod_{i \in V'} \frac{\Gamma(k_{i} + \delta_{i_1})}{\Gamma(k_{i} - 1+\Delta + \delta_{i_1})}\Big).
\end{aligned}
\end{equation}

\subsection{Proof of Theorem~\ref{th:firstfo0}}

We rely on the following result which is from Feller Vol II  problem 12 page 40
\begin{mdframed}[style=MyFrame1]
\begin{lemma}\label{le:Pdnk}
Fix a one to one function  $ f \colon \N \rightarrow (0, \infty)$.   Consider the collection $(X_i)_i$ of independent exponential random variables, where $X_i$ has rate $f(i)$. Then
\begin{equation}\label{eq:nicefor}
\P\Big(\sum_{i=1}^n X_i \in (t, t + {\rm d} t)\Big)  = (\prod_{k=1}^n f(k) ) \sum_{i=1}^n \Big( \prod_{\substack{j = 1 \\ j \ne i}}^n \frac{1}{f(j) - f(i)} \Big) e^{-f(i) t} {\rm d} t, \quad t \ge 0
\end{equation}
\end{lemma}
\end{mdframed}
\vspace{0.2cm}
Let $\X$ be a directed reinforced random walk with strictly increasing reinforcement function $f$. 
\vspace{0.2cm}
\begin{mdframed}[style=MyFrame1]
\begin{lemma}
$$
\P(S_n \in (0, \ell), S_{n+1} > \ell) = \Big(\prod_{k=1}^n f(k) \Big)\sum_{i=1}^n\Big(\prod_{\substack{j = 1 \\ j \ne i}}^{n+1} \frac{1}{f(j) - f(i)}\Big) ( {\rm e}^{-f(i) \ell} - {\rm e}^{-f(n+1) \ell}).
$$
\end{lemma}
\end{mdframed}
\begin{proof}
$$
\begin{aligned}
\P(S_n \in (0, \ell), S_{n+1} > \ell) &= \int_{0}^\ell \P(S_n \in (t, t+{\rm d} t)) {\rm e}^{-f(n+1) (\ell -t)} {\rm d} t\\
&= {\rm e}^{-f(n+1) \ell}   \int_{0}^\ell (\prod_{k=1}^n f(k) ) \sum_{i=1}^n \Big( \prod_{\substack{j = 1 \\ j \ne i}}^n \frac{1}{f(j) - f(i)} \Big) {\rm e}^{-(f(i)-f(n+1)) t}  {\rm d} t\\
&= \Big(\prod_{k=1}^n f(k) \Big)\sum_{i=1}^n\Big(\prod_{\substack{j = 1 \\ j \ne i}}^{n+1} \frac{1}{f(j) - f(i)}\Big) ( {\rm e}^{-f(i) \ell} - {\rm e}^{-f(n+1) \ell})
\end{aligned}
$$
\end{proof}
\subsection{Directed Once-Reinforced Random walks}
Define the incomplete gamma function $\gamma(s, x)$ by
\begin{equation}\label{eq:gammainc}
\gamma(s, x) = \int_0^x t^{s-1} e^{-t} \, dt = x^s e^{-x} \sum_{k=0}^{\infty} \frac{x^k}{\Gamma(s + k + 1)} = \sum_{k=0}^{\infty} \frac{x^s e^{-x} x^k}{s(s+1)\cdots(s+k)}.
\end{equation}
\begin{mdframed}[style=MyFrame1]
\begin{lemma}\label{le:poiss1} Let $X, Y, Z$ be  three independent random variables. We assume that 
\begin{itemize}[ leftmargin=*, itemsep=0pt, parsep=0pt, topsep=0pt, partopsep=0pt]
\item $X$ is an   exponential random variable with mean $1/a$, where $a \in (0, \infty)$.
\item  $Y$ is distributed as Gamma $(n-1, 1)$, with the convention that $Y = 0$ if $n = 1$.
\item $Z$ is an exponential with parameter one.
\end{itemize}
 Then
$$
\mathbb{P}(X + Y \in [\ell, \ell + d\ell]) = a e^{-a \ell} d\ell \cdot \frac{\gamma(n-1, (1 - a)\ell)}{\Gamma(n-1)(1 - a)^{n-1}}
$$
\begin{equation}\label{eq:PXY}
\begin{aligned}
\mathbb{P}(X + Y \leq \ell < X + Y + Z) &= \frac{a}{(1-a)^n(n-1)!} {\rm e}^{-a \ell}  \gamma(n, (1-a) \ell) =:g(n, a).
\end{aligned}
\end{equation} 
\end{lemma}
\end{mdframed}
\begin{proof}
 Conditioning on $X$ (alternatively using convolution), one has 
\begin{equation}\label{eq:PXY}
\begin{aligned}
\mathbb{P}(X + Y &\in  [\ell, \ell + d\ell]) = \int_0^\ell a {\rm e}^{- ax} \mathbb{P}(Y \in  [\ell -x, \ell -x + d\ell] ) {\rm d} x\\
&= \int_0^\ell a {\rm e}^{- ax}  {\rm e}^{- (\ell -x)} \frac{(\ell -x)^{n-2}}{(n-2)!}   {\rm d} \ell {\rm d} x\\
 &=  \frac{a}{(n-2)!}  {\rm e}^{-a \ell} \int_0^\ell {\rm e}^{-(1-a)(\ell -x)} (\ell -x)^{n-2}{\rm d} x {\rm d} \ell\\
 &=\frac{a}{(n-2)!}  {\rm e}^{-a \ell} \frac 1{(1-a)^{n-1}} \int_0^{(1-a)\ell} {\rm e}^{-z} z^{n-2} {\rm d}z{\rm d} \ell \qquad \qquad \mbox{(where $z = \ell -x$),}\\
 &= \frac{a}{(1-a)^{n-1}(n-2)!} {\rm e}^{-a \ell}  \gamma(n-1, (1-a) \ell) {\rm d} \ell, \qquad \qquad \mbox{where $\gamma$ is defined in \eqref{eq:gammainc}}. 
\end{aligned}
\end{equation}

In fact, one has 
$$
\begin{aligned}
\mathbb{P}(X + Y \leq \ell < X + Y + Z) &= \int_0^\ell a {\rm e}^{- ax} \mathbb{P}(Y \leq \ell -x <  Y + Z) {\rm d} x.
\end{aligned}
$$
The probability in the integrand equals to the probability that a Poisson process of rate one has exactly $n-1$ arrivals before time $\ell -x$. This is the probability mass function of a Poisson random variable with parameter $\ell -x$ evaluated at $n-1$.  Hence, 
$$
\begin{aligned}
\mathbb{P}(X + Y &\leq \ell < X + Y + Z) = \int_0^\ell a {\rm e}^{- ax} {\rm e}^{-(\ell -x)} \frac{(\ell -x)^{n-1}}{(n-1)!} {\rm d} x\\
&=  \frac{a}{(n-1)!}  {\rm e}^{-a \ell} \int_0^\ell {\rm e}^{-(1-a)(\ell -x)} (\ell -x)^{n-1}{\rm d} x\\
&=\frac{a}{(n-1)!}  {\rm e}^{-a \ell} \frac 1{(1-a)^n} \int_0^{(1-a)\ell} {\rm e}^{-z} z^{n-1} {\rm d}z\\
&= \frac{a}{(1-a)^n(n-1)!} {\rm e}^{-a \ell}  \gamma(n, (1-a) \ell). 
\end{aligned}
$$
\end{proof}
 \begin{mdframed}[style=MyFrame1]
\begin{lemma}\label{le:firstfo} Under the conditions of Theorem~\ref{eq:thD}, one has
\begin{equation}\label{eq:loctim}
\begin{aligned}
\mathbb{P}^{\ssup a} \big(&k(\X, \sigma) = k,\; \ell(\X, \sigma) \in (\ell, \ell+ d\ell),\; \vec{T}(\X, \sigma) = \vec{T} \big) \nonumber \\
&=  \frac{a^{|E'|}}{(1 - a)^{\|k\| -|\vec{T}|}} e^{-a \sum_{i \in V'} deg_i \times  \ell_i } \prod_{ij \in \vec{E}' }   \frac{\gamma(k_{i,j} - h_{i,j}(\vec{T}), (1 - a)\ell_i)}{\Gamma(k_{i,j} -h_{i,j}(\vec{T}))}  m_\sigma(d \ell),
\end{aligned}
\end{equation}
where $(h_{i,j}(\vec{T}))_{(i, j) \in \vec{E}'}$ where defined in \eqref{eq:defH}.
\end{lemma}
\end{mdframed}
\begin{proof}[Proof of Lemma \ref{le:firstfo}]
For each vertex $i \in V' \setminus \{i_1\}$ there exist an  edge $(i, j)$ which is traversed at the very last jump from $i$ in the interval $[0, \sigma]$, at time $\ell_i$. The number of jumps from $i$ to any of its neighbors $j$ in the interval $(0, \ell_i)$ is either $k_{i, j} -1$ if $(i, j)  \in \vec{T}$ or equals $k(i, j)$ otherwise.  We use Lemma \eqref{le:poiss1}.
\end{proof}


\section{Proof of Theorem~\ref{eq:thD}}
By a simple integration by parts, we can rewrite 
$$
\gamma(n, x) = \Gamma(n) {\rm e}^{-x} \sum_{k=n}^{\infty} \frac{x^k}{k!}.
$$
  Let $D$ be a geometric random variable with probability mass function 
$$
\P^{\ssup a} (D = n) = a (1-a)^n, \qquad \mbox{for $n \ge 0$}.
$$
One has
\begin{equation}\label{eq:intrge}
\begin{aligned}
 \frac{a}{(1-a)^n(n-1)!} {\rm e}^{-a \ell}  \gamma(n, (1-a) \ell) &= a {\rm e}^{-  \ell} \sum_{k=n}^\infty (1-a)^{k - n} \frac{\ell^k}{k!}= \E\left[{\rm e}^{-  \ell} \frac{\ell^{n +D}}{(n+D)!}  \right].
\end{aligned}
\end{equation}
Using \eqref{eq:intrge}, and setting $k'_{i,j} = k_{i,j} - h_{i,j}(\vec{T})$, where $(h_{i,j})_{(i, j) \in \vec{E}}$ where defined in \eqref{eq:defH}, one has
\begin{equation}\label{eq:idwithk}
 \frac{a}{(1-a)^{k'_{i,j}}(k'_{i,j}-1)!} {\rm e}^{-a \ell} \frac{\gamma(k'_{i,j}, (1 - a)\ell_i)}{\Gamma(k'_{i,j})} = \E\left[{\rm e}^{-  \ell} \frac{\ell^{k'_{i,j} +D}}{(k'_{i,j}+D)!}  \right].
  \end{equation}
By plugging \eqref{eq:idwithk} in Lemma~\ref{le:firstfo}, using the fact that $(D_{i,j})_{(i, j) \in \vec{E}}$ are i.i.d.  one has the following result.
\begin{mdframed}[style=MyFrame1]
\begin{corollary}
  \begin{equation}\label{eq:loctim1}
\begin{aligned}
\mathbb{P}^{\ssup a}_{0} \big(&k(X, \sigma) = k,\; \ell(X, \sigma) \in (\ell, \ell+ d\ell),\; \vec{T}(X, \sigma) = \vec{T} \big) \nonumber \\
&=   {\rm e}^{-  \sum_{i \in V} {\rm deg}_i \ell_i} \prod_{(i,j) \in \vec{E}'} \E^{\otimes D_{i,j}}\left[ \frac{\ell_i^{k_{i,j} + D_{i,j} - h_{i,j}(\vec{T})}}{(k_{i,j}+D_{i,j}- h_{i,j}(\vec{T}))!} \right]m_\sigma(d \ell),
\end{aligned}
\end{equation}
where $(h_{i,j}(\vec{T}))_{(i, j) \in \vec{E}'}$ where defined in \eqref{eq:defH}.
\end{corollary}
\end{mdframed}
\begin{proof}[Proof of Theorem~\ref{eq:thD}]
Let \( b \in \mathcal{I} \) be such that \( b_i = \delta_{i_0}(i) - \delta_{i_1}(i) \) for all \( i \in V \), where $\delta_{i}(j)$ is the usual dirac mass which equals one if and only if $j =i$, and it is zero otherwise. For each pair of neighbors $\{i, j\}$, choose a unique orientation  such that \( b_{i,j} \ge 0 \). Recall that $\vec{E}^+$ be the collection of these oriented edges. In particular, \( \vec{E}^+ = \{ (i, j)\in \vec{E}' \colon b_{i, j} \ge 0 \} \).

To compute the probability in the theorem statement, sum the contributions from Corollary \eqref{eq:loctim1}  over all \( k \in \mathbb{N}^{\vec{E}} \) such that \( b(k) = b \). For each \( (i, j) \in \vec{E}^+ \), we sum over \( k_{ji} \ge 0 \) with \( k_{ij} = k_{ji} + b_{ij} \ge k_{ji} \). Then, using Corollary \eqref{eq:loctim1},  and analysing together the contributions from \( (i, j) \) and \( (j, i) \) for each \( (i, j) \in \vec{E}^+ \),  using independence among $(D_{i,j})_{(i,j) \in \vec{E}}$, one gets 
$$
\begin{aligned}
&\mathbb{P}^{\ssup a} \left( b(k(X, \sigma)) = b,\, \ell(X, \sigma) \in (\ell, \ell + d\ell),\, \vec{T}(X, \sigma) = \vec{T} \right)\\
&= {\rm e}^{ -  \sum_{i \in V} {\rm deg}_i \ell_i } 
\sum_{ (k_{ji}) \in \mathbb{N}^{\vec{E}^+} }
\prod_{(i, j) \in \vec{E}^+}
\frac{(\ell_j)^{k_{j, i}}}{k_{j, i}!}
\frac{(\ell_i)^{k_{i, j}}}{k_{i, j}!}
\prod_{(i, j) \in \vec{T}} \frac{k_{ij}}{\ell_i} 
\prod_{i \in V} d\ell_i.
\end{aligned}
$$

Set \( k'_{j,i} := k_{j, i} - h_{j, i} (\vec{T}) \). Then
\[
k_{i,j} - h_{j, i} (\vec{T}) = k'_{j,i} + \tilde{b}_{i, j}, \quad
k_{i,j} + k_{j, i} = 2k'_{j, i} + \tilde{b}_{i, j} + \mathbf{1}_{i, j \in T}
\]
where \( T \) is the undirected tree associated to \( \vec{T} \). Therefore, 
$$
\begin{aligned}
&\sum_{k_{j, i} \ge h_{j,i}(\vec{T}) }
\frac{ \ell_i^{k_{i, j} + D_{i, j}- h_{i,j}(\vec{T})} \ell_j^{k_{j, i} + D_{j ,i} -h_{j, i}(\vec{T})}}
{(k_{i, j} +D_{i, j}  - h_{i,j}(\vec{T}))! (k_{j, i}+ D_{j ,i} -h_{j, i}(\vec{T}))!}\\
&= \sum_{k'_{j, i} \ge 0}
\frac{\ell_i^{k'_{j,i} + D_{i, j} + \tilde{b}_{i,j}} \ell_j^{k'_{j, i} + D_{j, i}}}
{(k'_{j,i} +D_{i, j}+ \tilde{b}_{i,j})! \, {(k'_{j, i} +  D_{j ,i})}!}\\
&= \ell_i^{\frac{\tilde{b}_{i, j} + D_{i,j} - D_{i,j}}2} \ell_j^{\frac{D_{j, i} - D_{i,j} + \tilde{b}_{j, i}}2} 
\sum_{k'_{ji} \ge 0}
\frac{( \sqrt{\ell_i \ell_j})^{2k'_{j, i} +  D_{i,j} + \tilde{b}_{i,j} + D_{j,i}}}{(k'_{j, i} + D_{i, j} + \tilde{b}_{i,j})! \, (k'_{j,i} + D_{j, i})!}.\\
\end{aligned}
$$
The proof ends by noting 
\[
\prod_{(i, j) \in \vec{E}^+} \ell_i^{\tilde{b}_{i, j}/2} \ell_j^{\tilde{b}_{j, i}/2}
= \prod_{i \in V} \ell_i^{\tilde{b}_i/2}.
\]
\end{proof}
\begin{mdframed}[style=MyFrame1]
\begin{theorem}\label{th:thDo} One has 
\begin{equation}\label{eq:loctim001}
\begin{aligned}
&\vec{\mathbb{P}}^{\ssup a} \big(\tilde{b}(\X, \sigma) = \tilde{b},\; \ell(\X, \sigma) \in (\ell, \ell+ d\ell),\; \vec{T}(\X, \sigma) = \vec{T} \big)  \\
&\le  
a^{|V'|}{\rm e}^{(1-a)\sum_{i \in V'} {\rm deg}_i \ell_i} \mathbb{P} \left( b(k(\X, \sigma)) = b,\, \ell(\X, \sigma) \in (\ell, \ell + d\ell),\, \vec{T}(\X, \sigma) = \vec{T} \right)
\end{aligned}
\end{equation}
\end{theorem}
\end{mdframed}
\begin{proof} One has
$$
\begin{aligned}
&\sum_{k_{j, i} \ge h_{j,i}(\vec{T}) }
\frac{ \ell_i^{k_{i, j} + D_{i, j}- h_{i,j}(\vec{T})} \ell_j^{k_{j, i} + D_{j ,i} -h_{j, i}(\vec{T})}}
{(k_{i, j} +D_{i, j}  - h_{i,j}(\vec{T}))! (k_{j, i}+ D_{j ,i} -h_{j, i}(\vec{T}))!}\le \frac{\ell_i^{D_{i,j}} \ell_j^{D_{j,i}}}{D_{i,j}! D_{j, i}!} \sum_{k'_{j, i} \ge 0}
\frac{\ell_i^{k'_{j,i}  + \tilde{b}_{i,j}} \ell_j^{k'_{j, i}}}
{(k'_{j,i} + \tilde{b}_{i,j})! \, {k'_{j, i} }!}\\
&= \frac{\ell_i^{\frac{\tilde{b}_{i, j}}2 + D_{i,j}}}{D_{i,j}!} \frac{\ell_j^{D_{j, i} + \frac{ \tilde{b}_{j, i}}2}}{D_{j,i}!}
\sum_{k'_{ji} \ge 0}
\frac{( \sqrt{\ell_i \ell_j})^{2k'_{j, i}  + \tilde{b}_{i,j} }}{(k'_{j, i}  + \tilde{b}_{i,j})! \, (k'_{j,i} )!}.= \frac{\ell_i^{\frac{\tilde{b}_{i, j}}2 + D_{i,j}}}{D_{i,j}!} \frac{\ell_j^{D_{j, i} + \frac{ \tilde{b}_{j, i}}2}}{D_{j,i}!} I_{ \widetilde{b}_{i, j}}(2 \sqrt{\ell_i \ell_j}).
\end{aligned}
$$
Hence 
$$
\begin{aligned}
&\mathbb{P}^{\ssup a} \left( b(k(\X, \sigma)) = b,\, \ell(\X, \sigma) \in (\ell, \ell + d\ell),\, \vec{T}(\X, \sigma) = \vec{T} \right)\\
&\le \E^{\otimes D_{i,j}}\left[ 
\prod_{i  \in V'}  \frac{\ell_i^{ \sum_{ j \sim i} D_{i,j} } }{\prod_{j \sim i} D_{i ,j}!} \right] \mathbb{P} \left( b(k(\X, \sigma)) = b,\, \ell(\X, \sigma) \in (\ell, \ell + d\ell),\, \vec{T}(\X, \sigma) = \vec{T} \right)
\end{aligned}
$$
Using independence,  combined with
$$
\E^{\otimes D_{i,j}}\left[\frac{\ell_i^{D_{i,j}}}{D_{i,j}!}\right] = a \sum_{n=0}^\infty (1-a)^n \frac{\ell_i^n}{n!} = a {\rm e}^{(1-a) \ell_i},
$$
one has
$$
\begin{aligned}
\E^{\otimes D_{i,j}}\left[ 
\prod_{i  \in V'}  \frac{\ell_i^{ \sum_{ j \sim i} D_{i,j} } }{\prod_{j \sim i} D_{i ,j}!} \right]   = a^{|V'|}{\rm e}^{(1-a)\sum_{i \in V'} {\rm deg}_i \ell_i}.
\end{aligned}
$$
\end{proof}

\subsection{Proof of Theorem~\ref{eq:LDP}}
\begin{proof}[Lower bound for Theorem~\ref{eq:LDP}]
The probability that all $D_{i,j} = 0$ is $a^{|\vec{E}|}$, which is of constant order. 
Hence, 
$$
\begin{aligned}
&\log \vec{\P}^{\ssup a}(\ell(X, t) \in (\ell, \ell + {\rm d}\ell)) \\
&\ge \log \E^{\otimes D_{i,j}}\left[{\rm e}^{-  \sum_{i} {\rm deg}_i \ell_i} \prod_{(i,j) \in \vec{E}^+} J_{D'_{i,j},  D^*_{i,j}}(2 \sqrt{\ell_i \ell_j})\left(\prod_{(i,j) \in \vec{E}^+}  \ell_i^{D_{i,j} -  \frac{D'_{i,j}}2 } \ell_j^{ D_{j,i}-  \frac{D'_{i,j}}2}\right) \prod_{i \in V} \ell_i^{\widetilde{b}_i } m_\sigma(d \ell)\prod_{(i,j) \in \vec{E}} \1_{D_{i,j}=0}\right]\\
&\ge -  \sum_{i \in V} {\rm deg}_i \ell_i + \sum_{(i,j) \in \vec{E}^+}\log  J_{0,  \widetilde{b}_{i,j}}(2 \sqrt{\ell_i \ell_j}) + {\rm Const}\\
&\ge {\rm Const}  - \sum_{i \in V} {\rm deg}_i \ell_i + \sum_{(i,j) \in \vec{E}^+}2 \sqrt{\ell_i \ell_j}\\
&= {\rm Const} - \sum_{i \in V} \sum_{j \sim i} (\sqrt{\ell_i} - \sqrt{\ell_j})^2.
\end{aligned}
$$
\end{proof}
\begin{proof}[Upper bound for Theorem~\ref{eq:LDP}] Is a direct consequence of Theorem~\ref{th:thDo}.
\end{proof}



\section{Proof of Theorem~\ref{th:main4}}
In this section, we consider oriented-ORRW $\X$ on $\Z^d$, with $d \ge 2$. We recall some notation. Set $\vec{\mathfrak{G}}_t = (\mathfrak{R}_{t}, \vec{ \mathfrak{C}}_{t})$ where  $\mathfrak{R}_{t}$ and $\vec{ \mathfrak{C}}_{t}$ are the vertex range and the (oriented) edge range, respectively, of $\X$ by time $t$.  Denote by $(\tau_i)_i$ the jump times of this process.~Fix the sequence  $u \colon \N \rightarrow (0, \infty)$ such that $u_N= o(N^{\theta})$, for some $\theta \in (0, 1/d)$.   
  Recall that $[n] = \{1, 2, \ldots, n\}$. 
   For $a, b \in \Z^d$, denote by $\langle a, b \rangle $ the usual inner product $\sum_{i =1}^d a_i b_i$. Fix $\eps \in (0, 1)$. Recall the following. 
 \begin{mdframed}[style=MyFrame1]
 \begin{definition}
Let 
 $ \mathcal{A}_N := \{{\rm Ball}(0, (1-\eps)u_N) \subset \vec{\mathfrak{G}}_{\tau_N} \subset {\rm Ball}(0, (1+\eps)u_N)\},$ where these are graph inclusions. 
Define 
\begin{eqnarray*}
\mathcal{A} &:=& \bigcup_{m>0} \bigcap_{n > m} \mathcal{A}_n\\
\mathcal{B} &:=& \Big\{\lim_{N \ti} N^{-\frac 12} \sum_{v \in \partial \mathfrak{R}_{\tau_N}} \1_{(v, v - e_1) \notin \vec{ \mathfrak{C}}_{\tau_N}} -\1_{(v, v + e_1) \notin \vec{ \mathfrak{C}}_{\tau_N} } =0\Big\}\end{eqnarray*}
\end{definition}
\end{mdframed}
Our task is to prove 
$  \P^{\ssup a}\Big(\mathcal{A} \cap \mathcal{B}  \Big) =0.$\\
{\bf Strong construction of ORRW($a$) on oriented $\Z^d$.} We attach to each oriented edge $e$  an independent Poisson
process $P(e)$ with rate one. We use these processes to generate the jumps of ORRW($a$) as follows.~®Let $\chi_i(e)$ be the inter-arrival times of the Poisson process $P(e)$.
Each exponential is recycled  until it is used to generate a jump. More precisely, suppose that $X_t =x$, and let 
$$\mathfrak{f}(x,y) := {\rm card}(s \in [0,t] \colon X_{s-} = x, X_s =y) +1.$$
 The first jump after time $t$ is towards $y$ if and only if 
$$(a^{-1}  +  (1-a^{-1})\1_{\mathfrak{f}(x,y)\ge 2}) \chi_{\mathfrak{f}(x,y)}(x, y) = \min_{z\colon z \sim x} (a^{-1}  +  (1-a^{-1})\1_{\mathfrak{f}(x,z)\ge 2})\chi_{\mathfrak{f}(x,z)}(x, z).$$

 The event $\mathcal{A}_N$ is determined by the behaviour of the process within ${\rm Ball}(0, 2 u_{N})$,  as $\eps <1$. Hence, we can consider directed  ORRW which is defined on the finite set ${\rm Ball}(0, 2 u_{N})$ and is coupled with the original one.~In this way, we can rely on recurrence property of the process defined on the finite set ${\rm Ball}(0, 2 u_{N})$. 
\subsection{A restricted process $\widetilde{\X}$ coupled with $\X$.} 
 \begin{mdframed}[style=MyFrame1]
 \begin{definition}
 For any  $N \in \N$, let  $\widetilde{\X}^{\ssup N}$ be an oriented-ORRW defined on  the graph induced  on ${\rm Ball}(0, 2 u_{N})$.~We generate the jumps of this process using the same exponentials  which were used for $\X$, in the way described above.~Hence,  $\widetilde{\X}^{\ssup N}$ and $\X$ are perfectly coupled  up to  the time when the two processes  reach the boundary of ${\rm Ball}(0, 2 u_{N})$.~On the event $\mathcal{A}_{N}$ the two processes $\X$ and $\widetilde{\X}^{\ssup N}$ coincide by time $\tau_N$, as the latter time would be less than the hitting time  of the boundary of  ${\rm Ball}(0, 2 u_{N})$.
 \end{definition}
 \end{mdframed}
  
For $n \ge k$,  let $\mathcal{D}^{\ssup n}_k(x)$ be  the event that  $\widetilde{X}^{\ssup {n}}_{\tau_k} =x$ and  at least one edge in  $\{(x, x+e_1), (x, x-e_1)\}$   has not been traversed by time $\tau_k$. More formally, 
  $$
 \begin{aligned}
 &\mathcal{D}^{\ssup n}_k(x) :=  \{\widetilde{X}^{\ssup {n}}_{\tau_k} =x\} \cap \{ |\{(x, x+e_1), (x, x-e_1)\}\setminus \vec{\mathfrak{C}}_{\tau_k}| \ge 1 \}. 
 \end{aligned}
 $$
   Set 
$$\mathfrak{Q}_n(x) :=  \sum_{\tau_{m} < \tau_n } \langle \widetilde{X}^{\ssup {n}}_{\tau_{m+1}} - \widetilde{X}^{\ssup {n}}_{\tau_m}, e_1\rangle  \1_{ \mathcal{D}^{\ssup n}_{\tau_m}(x)}.$$
  $\mathfrak{Q}_\infty(x)$   does not depend on the choice of $\widetilde{\X}^{\ssup n}$, as long as $2 u_{n}> \|x\|_2$.  Moreover, the collection  $(\mathfrak{Q}_\infty(x))_{x \in {\rm Ball}(0, 2 u_{n})}$ is composed by i.i.d.  random variables.~Here, we use the fact that each of the $\widetilde{\X}^{\ssup n}$ is recurrent.
 
   For any pair of sequences $(c_N)_N$ and $(d_N)_N$, we write  $c_N \ll d_N$ to denote $\lim_{N \ti} \frac{c_N}{d_N} =0$. Let $(b_N)_N$  be a  sequence  satisfying  $u_N^{d/2+\eta} \ll  b_N \ll \sqrt{N}$, where $\eta >0$ is chosen such that $d\theta/2 +\eta\theta < 1/2$.
We drop the superscript in the notation of the process $\widetilde{\X}$, as it will be clear from the context. 
  Let ${\rm Hitt}(x) := \inf\{s >0 \colon \widetilde{X}_s = x\}$. The random variables ${\rm Hitt}(x)$ and $\mathfrak{Q}_\infty(x)$ are independent.   Set  
  $\mathfrak{Q}_{n, m} := \sum_{x \in \mathfrak{R}_{\tau_m}} \mathfrak{Q}_{n}(x).$ 
  
\begin{mdframed}[style=MyFrame1]
\begin{lemma}\label{eq:stimadiQN}
$\sum_{N=1}^\infty \P^{\ssup a} (\mathcal{A}_{N} \cap \{|\mathfrak{Q}_{\infty, N}| > b_{N}\}) <\infty.$
\end{lemma}
\end{mdframed}
\begin{proof}
       Order the vertices  visited by the process chronogically using  $(x_i)_i,$ with ${\rm Hitt}(x_i) < {\rm Hitt}(x_{i+1})$. Observe that $\mathfrak{Q}_\infty(x_k)$ is independent of $x_k$.~We argue next that $(\mathfrak{Q}_\infty(x_k))_{k \in \N}$ is composed by  i.i.d. random variables symmetric around zero.~To see this,  set a finite ordered set of distinct indices $\{i_1, \ldots, i_k\} \in \N^{k}$, with $i_k$ being the largest index,  and use a backward recursion.~In fact, we have that $\mathfrak{Q}_\infty(x_k)$ is independent of $(\mathfrak{Q}_\infty(x_s))_{s < k}$.  Finally it is immediate to argue, by symmetry, that  $\mathfrak{Q}_\infty(x)$ has a symmetric distribution around zero. Moreover, the moment generating function (mgf) of  $\mathfrak{Q}_\infty(x)$ is finite in a neighbor of zero.
  On $\mathcal{A}_{N}$,
 \begin{equation}\label{eq:bounonII}
|\sum_{x \in \mathfrak{R}_{\tau_N}} \mathfrak{Q}_\infty(x)| =  |\sum_{i=1}^{  |\mathfrak{R}_{\tau_N}|} \mathfrak{Q}_\infty(x_i)| \le  \sup_{j \in [ M u_N^d]}  |\sum_{k=1}^j \mathfrak{Q}_\infty(x_k)|.
  \end{equation}
   Using Theorem \ref{th:Levyin}, in the Appendix,  combined with $b_N \gg u_N^{d/2 + \eps}$ and a standard Moderate Deviations Principle (see Theorem 3.7.1 on page 109 of \cite{dembo2009large}) we obtain the result.
   \end{proof}
   
       \begin{mdframed}[style=MyFrame1]
 \begin{definition}
  Let  $\mathcal{U}_{N}^{\ssup +}$  (resp. $\mathcal{U}_{N}^{\ssup -}$) be the number of vertices $v$  such that $(v, v+ e_1) \in \vec{\mathfrak{C}}_{\tau_N}$ while $(v, v- e_1) \notin \vec{\mathfrak{C}}_{\tau_N}$ (the other way around for $\mathcal{U}_{N}^{\ssup -}$).~Define $\mathcal{U}_{N}^{\ssup \bullet}$ to be the number of vertices $v$ in  $ \mathfrak{R}_{\tau_N}$  such that both  $(v, v+ e_1)$ and  $ (v, v- e_1)\notin \vec{\mathfrak{C}}_{\tau_N}$. On the event $\mathcal{A}_N$, the vertices that satisfy the conditions described above  lie in the annulus.
  \end{definition}
  \end{mdframed}
     
\begin{mdframed}[style=MyFrame1]
\begin{lemma} \label{le:bara1}  One has 
 \begin{eqnarray}
&&\sum_{N=1}^\infty  \P^{\ssup a}\left(\Big\{\Big|(\mathcal{U}^{\ssup +}_{N}  - \mathcal{U}^{\ssup -}_{N})(\delta-1) +  \mathfrak{Q}_{N, N}\Big|  >  2 b_N\Big\} \cap \mathcal{A}_{N}\cap \{|\mathfrak{Q}_{\infty, N}|  < b_{N}\} \right) <\infty. \label{eq:Qi}
\end{eqnarray}
\end{lemma}
\end{mdframed}
\begin{proof}
By adding and subtracting, 
$$
\mathfrak{Q}_{\infty, N} =  \Big(\mathfrak{Q}_{\infty, N}- \mathfrak{Q}_{N, N}\Big) +  \mathfrak{Q}_{N, N}. $$
For any vertex $x$, let 
$${\rm Sat}(x) := \inf\{t \ge 0 \colon  \{(x, x+e_1), (x, x-e_1)\} \subset \vec{\mathfrak{C}}_{t}\}.$$
In words, ${\rm Sat}(x)$is the first time both  oriented edges $(x, x+e_1)$ and $(x, x-e_1)$ have been traversed. \\
Suppose that   $|\{(x, x+e_1), (x, x-e_1)\}\setminus \vec{\mathfrak{C}}_{\tau_N}| =1$. For simplicity, suppose that $(x, x+e_1) \in \vec{\mathfrak{C}}_{\tau_N}$.
 The number of jumps  of  the (recurrent) process  $\widetilde{\X}$ from $x $ to $x+e_1$, in the time interval $(\tau_N, {\rm Sat}(x)]$,  is a geometric random variable $\gamma(x)$ with probability mass function
$$\P^{\ssup a}(\gamma(x) = j) = \left(\frac{\delta}{\delta+1}\right)^j \frac{1}{1+\delta}, \qquad j \ge 0,$$
where $\delta = 1/a$.
There  exists   i.i.d. geometric random $(\gamma^{+}_{s})_{s \in \N}$ and $(\gamma^{-}_s)_{s \in \N}$,  which share the same distribution of $\gamma(x)$ and  independent of $\mathcal{F}_{\tau_N}$, 
 such that  
\begin{equation}\label{eq:decompofQ}
 \mathfrak{Q}_{\infty, N}- \mathfrak{Q}_{N, N} = \Big(\sum_{s =1}^{\mathcal{U}^{\ssup +}_{N}} (\gamma^{+}_{s}-1) -  \sum_{w =1}^{\mathcal{U}^{\ssup -}_{N}} (\gamma^{-}_{w}-1)\Big) + \sum_{u =1}^{\mathcal{U}^{\ssup \bullet}_{N}} L_u,
\end{equation}
where $(L_u)_u$ is a sequence of   i.i.d. random variables which are symmetric around zero and their mgf is finite in a neighbor of 0.
In order to prove \eqref{eq:decompofQ}, notice that if the vertex $x \in \partial \mathfrak{G}_N$, and both $(x, x+e_1)$ and $(x, x- e_1)$ are traversed by time $\tau_N$, then $v$ gives no further contribution to $\mathfrak{Q}_{\infty, N}$ after time $\tau_N$.~The right-hand side of \eqref{eq:decompofQ} describes the contribution of the other three scenarios.
 Using \eqref{eq:decompofQ}, we obtain
\begin{equation}
\begin{aligned}\label{eq:Qn}
\mathfrak{Q}_{\infty, N}   = \left(\sum_{s =1}^{\mathcal{U}^{\ssup +}_{N}} (\gamma^{+}_{s}-1) -  \sum_{w =1}^{\mathcal{U}^{\ssup -}_{N}} (\gamma^{-}_{w}-1)\right)  + \sum_{u =1}^{\mathcal{U}^{\ssup \bullet}_{N}} L_u + \mathfrak{Q}_{N, N}. 
\end{aligned}
\end{equation}
As the mean of  the geometric $\gamma$ is $\delta$, it is convenient to add and subtract  $(\delta-1)(\mathcal{U}^{\ssup -}_{N}  - \mathcal{U}^{\ssup +}_{N})$.
\begin{equation}
\begin{aligned}\label{eq:Qn1}
\mathfrak{Q}_{\infty, N}&=   \mathfrak{Q}_{N, N}   + \Big(\sum_{s =1}^{\mathcal{U}^{\ssup +}_{N}} (\gamma^{+}_{s} -\delta) -  \sum_{s =1}^{\mathcal{U}^{\ssup -}_{N}} (\gamma^{-}_{s,i}- \delta)\Big) + \sum_{u =1}^{\mathcal{U}^{\ssup \bullet}_{N}} L_u+ \left(\mathcal{U}^{\ssup -}_{N}  - \mathcal{U}^{\ssup +}_{N}\right) (\delta-1)  \\
&=: \mathfrak{Q}_{N, N} +  R_N  +\left(\mathcal{U}^{\ssup -}_{N}  - \mathcal{U}^{\ssup +}_{N}\right) (\delta-1). \qquad \mbox{\color{blue} (Defines $R_N$)}. 
\end{aligned}
\end{equation}
On the event $\{|\mathfrak{Q}_{\infty, N}|  < b_{N}\}$ one has 
\begin{equation}\label{eq:piz1}
 \mathfrak{Q}_{N, N} +  R_N  +\left(\mathcal{U}^{\ssup -}_{N}  - \mathcal{U}^{\ssup +}_{N}\right) (\delta-1) < b_N. 
\end{equation}
We next show that $|R_N| < b_N$ with high probability on $\mathcal{A}_{N}$, and this  combined with \eqref{eq:piz1} will conclude the proof.
There exists a constant $M$ depending on $a$ and $d$ only, such that on  $\mathcal{A}_N$ we have   
$|\partial \mathfrak{R}_{\tau_N}|  \le \eps M  u^{d}_N,$
as the latter bounds the size of the annulus which contains the boundary of the range. Hence,
\begin{equation}\label{eq:bound on R}
|R_N| \le \sup_{j \le \eps M u_N^d} |\sum_{k=1}^j (\gamma^{+}_{k}-\delta)|   + \sup_{\ell \le \eps M u_N^d} |\sum_{k=1}^\ell (\gamma^{-}_{k}-\delta)| + \sup_{t \le \eps M u_N^d} |\sum_{k=1}^t L_t|,
\end{equation}
where $R_N$ was defined in \eqref{eq:Qn1}. 
Using a union bound we get  
  
\begin{equation}\label{eq:R1}
 \begin{aligned}
  &\P^{\ssup a} \big(\mathcal{A}_{N} \cap \{|R_N| > b_{N} \}\big)\le   2 \P^{\ssup a} \big(\mathcal{A}_{N} \cap \big\{\sup_{m \in [0, \eps M u^{d}_N]} |\sum_{s =1}^{m} (\gamma^{+}_{s}-\delta) | > \frac{b_{N}}3 \big\}\big)\\
  &+\P^{\ssup a} \big(\mathcal{A}_{N} \cap \big\{\sup_{m \in [0,  \eps M u^{d}_N]} |\sum_{u =1}^{m} L_u | > \frac{b_{N}}3 \big\}\big)
  \end{aligned}
 \end{equation}
 As $b_N \gg u_N^{d/2 +\eta}$, we can use Moderate Deviations Principle (see Theorem 3.7.1 on page 109 of \cite{dembo2009large}) for mean-zero,  i.i.d. random variables with finite mgf  in a neighbor of zero, to get a bound for the right-hand side of \eqref{eq:R1}. The latter bound  is summable, and this ends the proof.
\end{proof}
  We immediately have 
  $$
  \mathcal{U}^{\ssup +}_{N}  - \mathcal{U}^{\ssup -}_{N} =  \sum_{v \in \partial \mathfrak{R}_{\tau_N}} \1_{(v, v - e_1) \notin \vec{ \mathfrak{C}}_{\tau_N}} -\1_{(v, v + e_1) \notin \vec{ \mathfrak{C}}_{\tau_N} }.
  $$
Hence
\begin{equation}\label{eq:stuf1}
 \lim_{N \ti} \frac{\mathcal{U}^{\ssup +}_{N}  - \mathcal{U}^{\ssup -}_{N}}{\sqrt{N}} =0, \qquad \mbox{on the set $\mathcal{B}$}.
\end{equation}
Combine \eqref{eq:stuf1} with Lemma \ref{eq:stimadiQN} and Lemma \ref{le:bara1} to obtain
\begin{equation}\label{eq:bounQ0}
\lim_{N \ti} \frac{\mathfrak{Q}_{N,N}}{\sqrt{N}} = 0.
\end{equation}
\begin{mdframed}[style=MyFrame1]
 \begin{definition}
  Define, for $n \in \N$, 
  \begin{equation}\label{eq:SRW}
  S_{\tau_n} := \langle {X}_{\tau_n},  e_1\rangle - \sum_{x \in \mathfrak{R}_{\tau_n}} \mathfrak{Q}_{n}(x) = \langle {X}_{\tau_n},  e_1\rangle - \mathfrak{Q}_{n,n}.
  \end{equation}
   On $\mathcal{A}_n$, the process $(S_{\tau_n})_{n \in [N]}$  is a lazy one-dimensional simple random walk with finite time horizon, which coincides with the partial sums of the simple random walk steps of $(X_{\tau_i})_{i \in [N]}$.
     \end{definition}
  \end{mdframed}
\begin{proof}[Proof of Theorem~\ref{th:main4}]
  Let 
  $$\mathcal{C}_N := \Big\{{\rm Card}\{j \in [N]\colon  S_{\tau_{j+1}} -  S_{\tau_{j}} \neq 0\} \in [\frac{1-\eps}{d} N, \frac{1+\eps}{d} N]\Big\}\newconstant{co:co1}.$$
  One has 
  \begin{equation}\label{eq:bounB}\P^{\ssup a}(\mathcal{A}_N \cap \mathcal{C}_N^c)  \le {\rm e}^{-\useconstant{co:co1} \times u_N^d}, \qquad \mbox{for some $\useconstant{co:co1}>0$}.
  \end{equation}
  Let $ \mathcal{C} := \bigcup_{m>0} \bigcap_{n > m} \mathcal C_n$.
  On the event $\mathcal{A}\cap \mathcal{B}\cap  \mathcal{C}$, the process $(S_{\tau_N})_{N \in \N}$ satisfies a central limit theorem, and in particular
  \begin{equation}\label{eq:sfini1}
 \liminf_{N \ti} \P^{\ssup a}\left(\frac{S_{\tau_N}}{\sqrt{N}} >1\right) >0.
  \end{equation}
On the other hand, on $\mathcal{A}_N$, we have 
$$
| S_{\tau_N}  +  \mathfrak{Q}_{N,N} | =    |\langle X_{\tau_N}, e_1\rangle|  \le 4 u_N,
$$
where the last equality comes from the fact that the range is contained in the Ball($0, 2 u_N)$). Hence
\begin{equation}\label{eq:sfini2}
\lim_{N \ti} | N^{-\frac 12} S_{\tau_N} + N^{-\frac 12} \mathfrak{Q}_{N,N}|  =0, \qquad \mbox{on $\mathcal{A}$},
\end{equation}
as $u_N \ll \sqrt{N}$. Equations \eqref{eq:sfini2} combined with   \eqref{eq:bounQ0} contradicts \eqref{eq:sfini1}.
\end{proof}

\section{Appendix }
The following is a well-known result by Donsker-Varadhan (\cite{DV79}).
\begin{mdframed}
[style=MyFrame1]
\begin{theorem}\label{th:DonVar}[Donsker-Varadhan]
 Denote by $\lambda_d$ the principal eigenvalue of the operator $-\Delta/2$ on ${\rm Ball}(0, 1)$, with Dirichlet boundary conditions,
and let $\omega_d$ be the volume of ${\rm Ball}(0, 1)$.
$$
\lim_{t \ti} \frac 1{t^{d/(d+2)}} \log \E\left[{\rm e}^{- |\mathfrak{R}_t|}\right] = - \frac{d + 2}2
\left(\frac{
2\lambda_d}d \right)
\omega^{2/(2+d)}_d := - \psi_d<0. 
$$
\end{theorem}
\end{mdframed}
The following is a well-known inequality by Paul L\'evy (see, for example, Lemma 5 page 72 in \cite{ChowTeicher})
\begin{mdframed}
[style=MyFrame1]
\begin{theorem}[L\'evy's inequality] \label{th:Levyin}
 Let $(Y_i)_i$ a sequence of i.i.d. random variables with  median equal to zero. Let $a>0$. We have that 
 $$ 
 \P\left(\max_{j \in [m]} \left|\sum_{i=1}^j Y_i \right|\ge a\right) \le 2 \P\left(\left|\sum_{i=1}^m Y_i\right|\ge a\right).
 $$
\end{theorem}
\end{mdframed}
\noindent {\bf \sc Acknowledgements}  We  would like to thank Franz Merkl and Silke Rolles for a discussion on Fourier analysis of the  local time formula in the case of simple random walk. 
The authors would like to thank  Amir Dembo, Daniel Kious, Tim Garoni, Minh Nguyen,   Bruno Schapira and Peter Wildemann for helpful discussions.  The
work of the authors   was partially supported by ARC DP230102209 and the National Science Foundation of China (NSFC), Grant No. 11771293
\bibliography{ORRW-Arxiv}  
\bibliographystyle{alpha}
\end{document}